\documentclass[aap,preprint]{imsart}
\RequirePackage{amsthm,amsmath,amsfonts,amssymb}
\RequirePackage[numbers]{natbib}
\RequirePackage[colorlinks,citecolor=blue,urlcolor=blue]{hyperref}
\RequirePackage{graphicx}
\usepackage{aliascnt}
\usepackage{cleveref}
\usepackage{bbm}
\usepackage{mathabx}
\usepackage{bm} 

\usepackage{xcolor}
\usepackage{comment}

\usepackage{algorithm}
\usepackage{algpseudocode}

\startlocaldefs
\theoremstyle{plain}

\newtheorem{theo}{Theorem}[section]
\newtheorem*{theo*}{Theorem}

\newaliascnt{prop}{theo}
\newtheorem{prop}[prop]{Proposition}
\aliascntresetthe{prop}

\newaliascnt{crl}{theo}
\newtheorem{crl}[crl]{Corollary}
\aliascntresetthe{crl}

\newaliascnt{lm}{theo}
\newtheorem{lm}[lm]{Lemma}
\aliascntresetthe{lm}

\theoremstyle{definition}

\newtheorem{defi}{Definition}

\newtheorem{asmp}{Assumption}
\newtheorem*{asmp*}{Assumption}

\newtheorem{rmk}{Remark}
\newtheorem{example}{Example}

\crefname{theo}{Theorem}{Theorems}
\crefname{prop}{Proposition}{Propositions}
\crefname{defi}{Definition}{Definitions}
\crefname{crl}{Corollary}{Corollaries}
\crefname{lm}{Lemma}{Lemmas}
\crefname{example}{Example}{Examples}
\crefname{rmk}{Remark}{Remarks}
\crefname{asmp}{Assumption}{Assumptions}

\crefname{section}{Section}{Sections}
\crefname{algorithm}{Algorithm}{Algorithms}
\crefname{table}{Table}{Tables}
\crefname{figure}{Figure}{Figures}

\newcommand{\R}{\mathbb{R}}

\newcommand{\Z}{\mathbb{Z}}

\newcommand{\E}{\mathbb{E}}

\newcommand{\1}{\mathbbm{1}} 
\newcommand{\sC}{\mathcal{C}}
\newcommand{\sX}{\mathcal{X}}
\newcommand{\sP}{\mathcal{P}}

\newcommand{\sO}{\mathcal{O}}
\newcommand{\sA}{\mathcal{A}}
\newcommand{\sB}{\mathcal{B}}

\newcommand{\sN}{\mathcal{N}}

\newcommand{\sL}{\mathcal{L}}
\newcommand{\sE}{\mathcal{E}}
\newcommand{\sD}{\mathcal{D}}
\newcommand{\PP}{\mathbf{P}}

\newcommand{\eps}{\varepsilon}

\newcommand{\supp}{\mathrm{supp}}

\newcommand{\sign}{\mathrm{sign}}

\newcommand{\Rg}{\mathrm{Range}}
\newcommand{\Dom}{\mathrm{Dom}}

\renewcommand{\d}{\mathrm{d}} 

\newcommand{\Id}{\mathrm{Id}}

\newcommand{\gap}{\mathrm{Gap}}
\newcommand{\diag}{\mathrm{diag}}

\newcommand{\var}{\mathrm{Var}}

\endlocaldefs

\begin{document}

\begin{frontmatter}
\title{Hamiltonian dynamics for sampling on discrete spaces}
\runtitle{Hamiltonian dynamics for sampling on discrete spaces}

\begin{aug}
\author[A]{\fnms{Raphael}~\snm{Barboni}\ead[label=e1]{raphael.barboni@unibocconi.it}}
\author[A]{\fnms{Sebastiano}~\snm{Grazzi}\ead[label=e2]{sebastiano.grazzi@unibocconi.it}}
\author[A]{\fnms{Giacomo}~\snm{Zanella}\ead[label=e3]{giacomo.zanella@unibocconi.it}}
\address[A]{Department of Decision Sciences and BIDSA, Bocconi University, Via Roentgen 1, 20136, Milan, Italy\printead[presep={,\ }]{e1,e2,e3}}

\end{aug}

\begin{abstract}
We develop a general class of non-reversible Hamiltonian Monte Carlo dynamics on discrete state spaces. The method augments the discrete state with a continuous momentum variable and does not require a continuous embedding of the discrete state space. We establish conditions for invariance of the target distribution and study the resulting processes in terms of ergodicity, exponential contractivity, asymptotic variance and relaxation time. We then derive scaling limits on increasingly fine lattices and high-dimensional hypercubes. In both settings, suitable rescalings converge to Hamiltonian dynamics in continuous space, revealing a diffusive-to-ballistic speed-up over reversible samplers, even for heterogeneous target distributions where standard non-reversible methods become diffusive. The proposed dynamics can be simulated exactly in continuous time, given access to the target distribution at all neighbours of the current state. To reduce computational cost, we provide approximation schemes based on splitting, $\tau$-leaping and gradient approximations. We illustrate the proposed framework numerically on examples involving high-dimensional hypercubes, mixture models and permutations.
\end{abstract}



\end{frontmatter}


\section{Introduction}
\subsection{Hamiltonian Monte Carlo in continuous spaces}
Markov chain Monte Carlo (MCMC) methods are popular sampling algorithms, which simulate a Markov chain $X_1,X_2,\dots$ whose limiting distribution coincides with a prescribed measure $\pi$ on a space $\mathcal{X}$ (e.g.\ a Bayesian posterior measure). This allows one to  approximate expectations $\int g(x)\pi(d x)$ with ergodic averages $\frac{1}{N}\sum_{i=1}^N g(X_i)$, for any integrable function $g$; see \cite{roberts2004general} for an overview. 
When $\mathcal{X} = \R^n$ and $\pi$  has a density proportional to $\exp(-U(x))$ for a differentiable potential $U:\R^n \to \R$, state-of-the-art MCMC methods are based on \emph{gradient-based} discretizations of \emph{non-reversible} continuous-time dynamics, with the most prominent examples being \emph{underdamped Langevin} and \emph{Hamiltonian} dynamics. Both these dynamics are defined in the product space of position and velocity and can be described by the infinitesimal generator
\begin{equation}
    \label{eq:generator_HMC_and_ULD}
    \Hat{\sL}^{ULD} f(x,p) = \langle p,  \nabla_x f(x,p)\rangle - \langle \nabla U(x),  \nabla_p f(x,p)\rangle 
    +\gamma \Hat{\sL}^{OU}_p f(x,p)
\end{equation}
for suitable test functions $f$, 
where 
$
\Hat{\sL}^{OU}_p f(x,p) = - \langle p, \nabla_p f(x,p) \rangle + \Delta_p f(x,p)
$
is the generator of the Ornstein-Uhlenbeck process on the velocity component $p$,
 $\Delta_p f$ is the Laplacian of $f$ with respect to $p$ and $\gamma\ge0$ is a friction (or damping) parameter.
 When $\gamma = 0$, \cref{eq:generator_HMC_and_ULD} reduces to Hamiltonian dynamics. 

The superiority of \eqref{eq:generator_HMC_and_ULD} over other popular reversible dynamics is apparent, for example, when the target distribution $\pi$ is a log-concave distribution with condition number $\kappa$.
In this case, the complexity of standard algorithms based on reversible dynamics scales as $\sO(\kappa)$  when $\kappa\to\infty$.
On the other hand, algorithms based on the non-reversible dynamics in \cref{eq:generator_HMC_and_ULD}, such as the \emph{Hamiltonian Monte Carlo (HMC)} algorithm, can potentially achieve a better dependence of
$\sO(\sqrt{\kappa})$. This striking improvement is often referred to as \emph{diffusive-to-ballistic} speed-up, see for example \cite{altschuler2025shifted, li2026space, lu2026sharp, gouraud2025hmc} for recent theoretical developments. 

More broadly, non-reversible algorithms have attracted attention from both theorists and practitioners as promising methods for accelerating the convergence of Monte Carlo estimators, in particular for ill-conditioned target distributions \cite{diaconis2000analysis, gagnon2024asymptotic, ascolani2025fast, sun2010improving, neal2004improving, bierkens2016non, bierkens2023sticky}.

\subsection{Sampling in discrete spaces}

Sampling problems in discrete spaces are ubiquitous in science and find applications, for example, in machine learning (e.g.\ binary classification problems and generative models  \cite{pakman2013auxiliary, nishimura2020discontinuous, grathwohl2021oops, bach2025sampling}),  statistical physics (e.g.\ for spin models \cite{gissler2026adjusted}) and Bayesian inference \cite{koskela2022zig, caron2012bayesian}. 

A difficulty encountered in such sampling problems is the lack of a natural geometric structure of the state space.
Indeed, the efficiency of HMC algorithms strongly relies on the differential structure of the ambient Euclidean space $\R^n$ and on the access, at low computational cost, to the gradient $\nabla U$ of the target potential.
In discrete spaces, however, analogous geometric structures are often missing, and gradients are not directly available.
Designing efficient sampling algorithms for general discrete spaces is therefore a nontrivial problem, although several partial solutions have been proposed in the literature; see \cref{subsec:related_works} below for a review.

\subsection{Overview of the paper}

The purpose of this work is to derive a framework for sampling from a broad class of distributions defined on discrete spaces. 
In particular, inspired by \cref{eq:generator_HMC_and_ULD}, we seek to define non-reversible dynamics in discrete spaces that do not require an embedding and are directly applicable to a range of complex discrete problems. In this setting, the number of existing general-purpose algorithms is limited.

We describe in~\cref{sec:discrete_hamiltonian_dynamics} a class of non-reversible Markov processes for sampling a given target distribution $\pi$ on a generic discrete space $\sX$.
These are based on the augmentation of the state space $\sX$ with a continuous velocity variable in $\R^d$.
They correspond to \emph{piecewise deterministic Markov processes (PDMPs)}~\cite{davis1984piecewise}, which we call \emph{discrete Hamiltonian dynamics}, to which an additional velocity refreshment process can be added. 
Importantly, defining these dynamics does not require the existence of an embedding of $\sX$ into $\R^d$.
It is sufficient to have the knowledge of a reversible Markov generator $\Bar{Q}$ on $\sX$, as well as a notion of direction for each of the possible transitions.

We first establish in~\cref{sec:theory} theoretical properties of the discrete Hamiltonian dynamics. We show ergodicity (\cref{prop:ergodicity}) and exponential contractivity (\cref{prop:spectral_gap}) under standard assumptions.
We also compare the discrete Hamiltonian dynamics to the associated reversible process: we show they always improve in terms of asymptotic variance (\cref{prop:asymptotic_variance}) and obtain lower bounds on their relaxation time (\cref{prop:relaxation_time}).

We then study in~\cref{sec:scaling_limits_infill,sec:scaling_limit_hypercube} scaling limits of the discrete Hamiltonian dynamics for two classes of targets: the discretization of a density on a square lattice in $\R^d$ (\cref{thm:infill_scaling_limit}) and an ill-conditioned and heterogeneous distribution on the hypercube $\{0,1\}^n$ (\cref{theo:scaling_limits_constrained_hypercube}).
In both cases, we recover randomized Hamiltonian Monte Carlo dynamics as a continuous limit.
More generally, these scaling limits suggest a diffusive-to-ballistic speed-up w.r.t. reversible samplers, even in regimes where some standard non-reversible algorithms lose their ballistic behaviour. A similar speed-up is also observed numerically in~\cref{sec:numerics} across qualitatively different examples.

Finally, we discuss in~\cref{sec:implementation} implementation details of a \emph{discrete HMC algorithm} based on discrete Hamiltonian dynamics.
An exact implementation is proposed in~\cref{alg:HMC_continuous_momentum}, exploiting the PDMP structure of the process.
Such an ``informed'' algorithm, however, requires the evaluation of the target density on the whole neighbourhood of the current state at each step, which can be computationally expensive.
We therefore propose approximation schemes that can be used to obtain faster algorithms, based on time-discretization, $\tau$-leaping or gradient approximation.
The computational cost–accuracy trade-off of these approximations is illustrated numerically in~\cref{sec:numerics}.

\subsection{Comparison with previously proposed non-reversible samplers}
\label{subsec:related_works}

\subsubsection{Methods based on continuous embeddings}\label{subsec:methods_embedding}
Continuous embeddings underpin several gradient-based MCMC methods for discrete spaces.
Building on informed proposals \cite{zanella2020informed} and gradient-based approximations of target ratios \cite{grathwohl2021oops}, several works develop discrete analogues of Langevin samplers \cite{grathwohl2021oops, zhang2022langevin,bach2025sampling,gissler2026adjusted}.

Several works have also derived HMC-like algorithms for discrete sampling by embedding the discrete state space (e.g.\ $\mathbb{Z}^n$ or $\{0,1\}^n$) into the Euclidean space $\mathbb{R}^n$~\cite{pakman2013auxiliary,mohasel2015reflection,nishimura2020discontinuous,zhou2020mixed}. These embeddings produce a target density in $\R^n$ that is piecewise constant, or more generally piecewise continuous. 
This enables the use of Hamiltonian dynamics on $\mathbb{R}^n$, provided suitable numerical integrators are employed to handle discontinuity surfaces. A potential drawback of this approach is that the resulting dynamics use information from $\pi$ only when hitting a discontinuity 
(since the gradient of the resulting distribution in $\mathbb{R}^n$ is either zero or independent of $\pi$ outside of the discontinuity surfaces). 

These methods are conceptually different from ours, as they fundamentally rely on a continuous embedding of the state space.
Consequently, their applicability is tied to settings where such an embedding is natural, and the admissible transition directions are induced by the embedding itself.
By contrast, our construction is intrinsically defined on the discrete state space and naturally accommodates arbitrary graphs and direction sets $\sigma_{x,y}$ which can be defined by the user using a priori knowledge of the problem.

In the particular case $\sX=\{0,1\}^n$, with $\sigma_{x,y}$ defined as in \cref{eq:n_momentum_hypercube} below, the resulting continuous-time dynamics share some similarities with ours, although both the underlying process and the numerical discretization are substantially different.
In particular, \cite{nishimura2020discontinuous} introduces a Laplace momentum together with a numerical integrator that exactly preserves the target distribution.
This yields an uninformed (or ``zeroth-order'') integrator requiring only a single target evaluation per integration step.
The algorithm scans through coordinates sequentially and includes a Metropolis rejection step, in a Metropolis-within-Gibbs fashion. By contrast, \cref{alg:HMC_continuous_momentum} is ``informed'', i.e., it requires $\mathcal{O}(D)$ target evaluations per iteration, where $D$ denotes the degree of the graph associated with the discrete space. Moreover, the algorithm is rejection-free and the $x$-component jumps to a new state at each iteration.

\subsubsection{Methods based on PDMPs}
Our discrete Hamiltonian dynamics correspond to a \emph{piecewise deterministic Markov process (PDMP)}~\cite{davis1984piecewise} and thus share similarities with other stochastic processes in this class.
Such PDMPs have recently formed the basis of several MCMC algorithms such as the \emph{Zig-Zag process}~\cite{bierkens2019zig} or the \emph{Bouncy Particle Sampler}~\cite{bouchard2018bouncy}.
These samplers are primarily designed for sampling smooth distributions in a continuous state space, which can be augmented by a discrete momentum variable, as is the case for the Zig-Zag process.
Our construction can instead be viewed as the converse: we augment an intrinsically discrete state space with a continuous momentum, yielding a PDMP whose deterministic evolution occurs in the momentum variable while the position evolves through stochastic jumps.
Implications of this change of perspective are discussed in detail in~\cref{sec:comparison_pdmps}. 
There has also been recent interest in developing MCMC algorithms based on PDMPs for sampling piecewise smooth densities, with a variety of applications~\cite{chevallier2024pdmp, bierkens2023methodsapplicationspdmpsamplers}.

\subsubsection{Methods based on lifting and skew-reversibility}

Another fruitful line of research on the design of non-reversible dynamics for discrete spaces is based on the notion of skew-reversibility, which relies on measure-preserving involutions (see e.g.~\cite[][Section~2]{andrieu2021peskun} or~\cite[][Section~2.2]{jansson2025rebalancing}).
This is for example the case of the skew-reversible Metropolis-Hastings algorithm, as defined in~\cite{bierkens2016non, diaconis2000analysis}, or of the Tabu sampler proposed by~\cite{power2019accelerated}.
Several works have in particular established the superiority of those skew-reversible samplers over reversible samplers in terms of asymptotic variance~\cite{andrieu2021peskun,gagnon2024asymptotic, diaconis2000analysis}.

A possible limitation of the skew-reversible Metropolis-Hastings algorithm is that it is forced to flip momentum at each rejection step.
This can lead to undesired diffusive behaviour as soon as the rejection rate is not close to zero. See, e.g., ~\cref{sec:scaling_limit_hypercube} and \cite{gagnon2024asymptotic,
roberts2025quantifying} for more details.
Also, skew-reversible samplers are usually based on an augmentation of the state space with a discrete velocity component $p \in \{-1, +1\}$.
This accounts for a one-dimensional momentum, which may be unsuitable for some applications.
In contrast, the discrete HMC algorithm discussed here has a multidimensional and continuous momentum, which we qualitatively expect to enhance ballistic behaviour.

\section{Discrete Hamiltonian Dynamics: a mixed space approach}\label{sec:discrete_hamiltonian_dynamics}

\subsection{Defining the dynamics}

In this work, we are concerned with the problem of sampling from probability distributions on a discrete state space $\sX$.
We denote by $\sP(\sX)$ the set of probability distributions on $\sX$.
In analogy with the dynamics in \cref{eq:generator_HMC_and_ULD}, we augment the state space with a $d$-dimensional momentum $p \in \R^d$, for some $d \ge 1$, and look for a non-reversible process on $\sX \times \R^d$ that is invariant w.r.t.\ the product measure
\begin{align} \label{eq:pi_hat}
    \Hat{\pi} \coloneq \pi \otimes \rho \, \in \, \sP(\sX \times \R^d),
\end{align}
where $\pi \in \sP(\sX)$ is some ``target'' distribution on $\sX$ and $\rho \in \sP(\R^d)$ is some distribution on the momentum space, such as a standard $d$-dimensional Gaussian.

\subsubsection{Discrete Hamiltonian dynamics}

We consider in this paper continuous-time Markov processes $Z_t = (X_t, P_t)$ on $\sX \times \R^d$ with generators of the form
\begin{align} \label{eq:discrete_hamiltonian_generator}
    \Hat{\sL}^H f(x,p) = \sum_{y \in \sX} Q_{x,y}(p) f(y,p) + \langle \mu_x(p), \nabla_p f(x,p) \rangle \,,
\end{align}
for (sufficiently regular) test functions $f:\sX \times \R^d \to \R$. 
In the above, $\mu : \sX \times \R^d \to \R^d$ is some ``drift'' function and, for every $p \in \R^d$, $Q(p) \in \R^{\sX \times \sX} $ is the generator of a continuous-time Markov chain on $\sX$~\cite{norris1998markov}, i.e., a $\sX \times \sX$ matrix satisfying:
\begin{align*}
    \begin{cases}
        Q_{x,y}(p) \geq 0, & x \neq y \in \sX \\
        Q_{x,x}(p) = - \sum_{y \neq x} Q_{x,y}(p), & x \in \sX .
    \end{cases}
\end{align*}
Equivalently, we say that $Z_t = (X_t, P_t)_{t \geq 0}$ solves the stochastic differential equation (SDE)
\begin{align} \label{eq:discrete_hamiltonian_SDE}
    \begin{cases}
        \d X_t &= \langle Q_{X_t}(P_t),  N(\d t) \rangle \,, \\
         \d P_t &= \mu_{X_t} (P_t) \d t\,, 
        \end{cases}
\end{align}
where  
$N$ is a $|\sX|$-dimensional Poisson measure \cite{applebaum2009levy}.  
Intuitively, the dynamics evolve as follows: at each instant, $X_t$ jumps to a new state $y\in\sX$ different from $X_t$ with rate $Q_{X_t,y}(P_t)$, 
while the momentum $P_t$ follows an ordinary differential equation with drift $\mu_{X_t}(P_t)$.

\subsubsection{Additional momentum refresh}
Similarly to the underdamped Langevin dynamics on $\R^d$, it is natural to introduce additional dynamics on the momentum $P_t$, such as the Ornstein-Uhlenbeck process considered in \cref{eq:generator_HMC_and_ULD}.
This leads to the generator
\begin{align} \label{eq:discrete_HMC_generator}
    \Hat{\sL} = \Hat{\sL}^H + \gamma\Hat{\sL}^D_p ,
\end{align}
where 
$\Hat{\sL}^D_p = \Id \otimes \sL^D_p$ with $\sL^D_p$ being the generator of some (typically $\rho$-reversible) Markov process on $\R^d$. In particular, we consider complete refreshments of the velocity component at exponential times with rate $\gamma \geq 0$, similarly to what is often done in Hamiltonian Monte Carlo (HMC) algorithms.
This corresponds to the refreshment operator
\begin{equation}
\label{eq:velocity_refreshment}
    \Hat{\sL}^D_p f(x,p) =  \int_{\R^d} (f(x, p') - f(x,p)) \d \rho(p').
\end{equation}
In this work, we will refer to Markov processes with generator of the form $\Hat{\sL}$ as in~\cref{eq:discrete_HMC_generator} and $\Hat{\sL}^D_p$ as in~\cref{eq:velocity_refreshment} as \emph{discrete Hamiltonian dynamics}, since they converge to the classical Hamiltonian dynamics in continuous space when rescaling space and time appropriately and for appropriate choices of $\mu$ and $Q$; see~\cref{sec:scaling_limits_infill,sec:scaling_limit_hypercube} for details.

\subsubsection{Existence and uniqueness}
Standard results ensure, under mild assumptions on $Q$ and $\mu$, the existence and uniqueness of a solution to the martingale problem associated with the generator in~\cref{eq:discrete_HMC_generator}  (see e.g.~\cite[][Theorem~1.19]{oksendal2007applied}, \cite[][Section~V]{protter2012stochastic} or~\cite[][Theorem~3.14]{sobczyk2013stochastic}).
In particular,
the solution to \cref{eq:discrete_hamiltonian_SDE} corresponds to a \emph{piecewise deterministic Markov process (PDMP)}~\cite{davis1984piecewise}.
In the rest of this paper, we will assume the existence and uniqueness of a Markov process $Z_t = (X_t, P_t)_{t \geq 0}$ whose generator is given by~\cref{eq:discrete_HMC_generator}.

\subsection{Invariance}

A necessary condition for a discrete Hamiltonian process to converge to the correct target distribution $\Hat{\pi}$ is that it must leave this distribution invariant.
\Cref{thm:invariance_solution} gives conditions on the generator $Q$ and the drift $\mu$ under which this is the case.
In the following, we will assume that the state space $\sX$ is finite and that both $p \mapsto Q(p) \in \R^{\sX \times \sX}$ and $p \mapsto \mu(p) \in (\R^d)^{\sX}$ are (component-wise) in $L^2(\rho)$.
It will also be convenient to assume that there exists a space $\sC$ of sufficiently smooth functions that is a \emph{core} for the generator $\Hat{\sL}$. The following assumption can be shown to hold under appropriate assumptions on the generator $Q$ and the drift function $\mu$; see \cite{durmus2021piecewise}.

\begin{asmp} \label{ass:core}
    The generator $\Hat{\sL}$ in~\cref{eq:discrete_HMC_generator} has domain $\Hat{\sD} \subset L^2(\Hat{\pi})$ and the set \hbox{$\sC^\infty_c(\sX \times \R^d)$} of smooth compactly supported functions is a core for this operator.
\end{asmp}

Under~\cref{ass:core}, the problem of verifying $\Hat{\pi}$-invariance is equivalent to verifying \emph{infinitesimal} $\Hat{\pi}$-invariance for test functions in $\sC^\infty_c(\sX \times \R^d)$~\cite{ethier2009markov}, i.e.,
\begin{align} \label{eq:infinitesimal_invariance}
    \E_{\Hat{\pi}} \Hat{\sL} f & = 0 , &  f \in \sC^\infty_c(\sX \times \R^d) .
\end{align}

\begin{theo} \label{thm:invariance_solution}
    Let $\Hat{\pi}$ be as in~\cref{eq:pi_hat} and    $\Hat{\sL}$ be as in \cref{eq:discrete_HMC_generator} satisfying~\cref{ass:core}. 
    Assume $Q_{x,y}(p) = 0$ when $\{ x, y \} \not\subset \supp(\pi)$ and for $(x,p) \in \sX \times\R^d$, define
    \begin{align} \label{eq:q}
        q_x(p) & \coloneq \sum_{y \in \sX} Q_{y,x}(p) \frac{\pi_y }{\pi_x}
        = \sum_{y \neq x} \left[ Q_{y,x}(p) \frac{\pi_y }{\pi_x} - Q_{x,y}(p) \right] ,
    \end{align}
    with the convention that $q_x(p) = 0$ for $x \notin \supp(\pi)$.
    Then the process generated by $\Hat{\sL}$ is $\Hat{\pi}$-invariant if and only if $\nabla_p \cdot ( \rho \mu_x) = \rho q_x$ in the distributional sense for $\pi$-a.e. $x \in \sX$, i.e.,
    \begin{align} \label{eq:invariance_condition}
         & \int_{\R^d} \langle \nabla_p f, \mu_x \rangle \d \rho = - \int_{\R^d} f q_x \d \rho ,
         &  f \in \sC^\infty_c(\R^d) .
    \end{align}
\end{theo}

\begin{proof}
    Note that, since $\Hat{\sL} = \Hat{\sL}^H + \gamma \Id \otimes \sL^D_p$ with $\sL^D_p$ the generator of a $\rho$-invariant Markov process, it suffices to study the $\Hat{\pi}$-invariance of $\Hat{\sL}^H$.
    Applying~\cref{eq:infinitesimal_invariance}, we see that $\Hat{\pi}$-invariance holds if and only if, for any $f \in \sC^\infty_c(\sX \times \R^d)$,
    \begin{align*}
        0 & = \int_{\sX \times \R^d} \Hat{\sL}^H f(x,p) \d \Hat{\pi}(x,p) \\
        & = \sum_{x \in \sX} \int_{\R^d} \left( \sum_{y \in \sX} Q_{x,y}(p) f(y,p) + \langle \mu_x(p), \nabla_p f(x,p) \rangle \right) \pi_x \d \rho(p) \\
        & = \sum_{x \in \sX} \pi_x \int_{\R^d} \left( f(x,p) q_x(p) + \langle \mu_x(p), \nabla_p f(x,p) \rangle \right) \d \rho(p) ,
    \end{align*}
    where we interchanged the sum and the integral in the last line using the integrability assumption on $Q$ and used the definition of $q_x$.
    Hence, it follows that~\cref{eq:invariance_condition} is a necessary and sufficient condition.
\end{proof}

\begin{rmk}
    In case $\rho$ has a positive density and there exists a smooth solution $u_x$ to the Poisson equation $\Delta u_x = \rho q_x$, solutions to~\cref{eq:invariance_condition} take the form:
    \begin{align*}
        \mu_x = \rho^{-1} \left( \nabla u_x + w_x \right),
        \quad \text{where} \quad 
        \nabla \cdot w_x = 0.
    \end{align*}
    In particular, the solution to~\cref{eq:invariance_condition} is unique only up to a divergence-free vector field.
\end{rmk}

\subsection{Construction of $\Hat{\pi}$-invariant discrete Hamiltonian dynamics} \label{subsec:factorized_generator}

We now give examples of generators $Q$ and drifts $\mu$ for which the generator $\Hat{\sL}$ in~\cref{eq:discrete_HMC_generator} is $\Hat{\pi}$-invariant.
We assume that a $\pi$-reversible generator $\Bar{Q} \in \R^{\sX \times \sX}$ is given, i.e., a generator satisfying the \emph{detailed-balance} condition
\begin{align*}
    \pi_x \bar Q_{x,y} &= \pi_y  \bar Q_{y,x}, &x \neq y \in \sX .
\end{align*}
For simplicity we will also always assume that $\Bar{Q}_{x,y} = 0$ when $\{x,y\} \not\subset \supp(\pi)$; see~\cref{sec:rev_kernels} for some examples of $\pi$-reversible generators $\Bar{Q}$.
We focus here on the case where the generator $Q$ can be written in the following factorized form:
\begin{align} \label{eq:Q_factorized}
    Q_{x,y}(p) =
    \begin{cases}
        \bar{Q}_{x,y} H_{x,y}(p) & \hbox{if $x \neq y$ and $\Bar{Q}_{x,y} > 0$,} \\
        0 & \hbox{if $x \neq y$ and $\Bar{Q}_{x,y} = 0$,} \\
        - \sum_{y \neq x} Q_{x,y}(p) & \hbox{if $x = y$,}
    \end{cases}
\end{align}
where, for $\Bar{Q}_{x,y} > 0$, $H_{x,y} : \R^d \to \R_+$ is some non-negative weight function. In this setting, the following corollary of \cref{thm:invariance_solution} shows that $\Hat{\pi}$-invariance is satisfied for well-chosen weight functions.
In particular, the drift function $\mu$ can be chosen to be a constant function of $p$, leading to easily and exactly integrable dynamics on the velocity in between state jumps.

\begin{crl} \label{crl:constant_drift}
    Assume $\rho$ has a positive density (also denoted by $\rho$) with distributional derivative s.t. $\nabla \log \rho \coloneq \rho^{-1} \nabla \rho \in L^2(\rho)$.
    Let $Q$ be of the form in~\cref{eq:Q_factorized} and assume that, for $\Bar{Q}_{x,y} > 0$, $H_{x,y} \in L^2(\rho)$ and there exists $\sigma_{x,y} \in \R^d $ such that $\sigma_{x,y} = - \sigma_{y,x}$ and such that
    \begin{align} \label{eq:constant_drift_condition}
        H_{x,y}(p) - H_{y,x}(p) &= - \langle \sigma_{x,y}, \nabla \log\rho(p) \rangle 
        & \text{for $\rho$-a.e. $p \in \R^d$.}
    \end{align}
    Then a solution to~\cref{eq:invariance_condition} is given by
    \begin{align} \label{eq:constant_drift}
    \mu_x(p) &= \sum_{y \neq x} \Bar{Q}_{x,y} \sigma_{x,y}, &(x,p) \in \sX \times \R^d.
    \end{align}
    For example, one solution to~\cref{eq:constant_drift_condition} is given by 
    \begin{align} \label{eq:minimal_H}
        H_{x,y}(p)=\max(0, - \langle \sigma_{x,y}, \nabla \log\rho(p) \rangle) .
    \end{align}
\end{crl}

\begin{proof}
    By~\cref{eq:q}, the detailed balance condition on $\Bar{Q}$ and the definition of $H_{x,y}$, we have for every $(x,p) \in \supp(\pi) \times \R^d$:
    $$
    q_x(p) = \sum_{y \neq x} \Bar{Q}_{x,y} \left[ H_{y,x}(p) - H_{x,y}(p) \right] = \sum_{y \neq x} \Bar{Q}_{x,y} \langle \sigma_{x,y}, \nabla \log\rho(p) \rangle.
    $$
    Hence the invariance condition~\cref{eq:invariance_condition} is readily solved by considering $\mu$ as in~\cref{eq:constant_drift}.
\end{proof}

\begin{rmk}
    In practice, the state space $\sX$ will be provided with a graph structure $(\sX, E)$ such that, for $x,y \in \sX$, there is an edge $(x,y) \in E$ if $\Bar{Q}_{x,y} > 0$.
    Intuitively, the family $(\sigma_{x,y})_{x,y \in \sX}$ here gives a notion of orientation to the edge $(x,y)$: from~\cref{eq:minimal_H}, a transition from $x$ to $y$ is only possible when $\langle \sigma_{x,y}, \nabla \log \rho(p) \rangle < 0$.
    We will give more details in \cref{sec:examples_graphs}.
    
    Finally, note that $H_{x,y}$ and $\sigma_{x,y}$ need only be defined when $\Bar{Q}_{x,y} > 0$.
    However, for convenience, we will often write sums involving product terms such as $\Bar{Q}_{x,y} \sigma_{x,y}$ or $\Bar{Q}_{x,y} H_{x,y}$ for $y \neq x$. By convention, these terms are set to $0$ if $\Bar{Q}_{x,y} = 0$.
\end{rmk}

\begin{rmk}
    Note that solutions to~\cref{eq:constant_drift_condition} are not unique.
    Indeed, general solutions take the form:
    \begin{align*}
        H_{x,y}(p) = \max(0,- \langle\sigma_{x,y}, \nabla \log \rho(p) \rangle) + \eps_{\{x,y\}}(p) ,
    \end{align*}
    where, for an (unordered) pair $x,y \in \sX$, $\eps_{\{x,y\}} : \R^d \to \R_+$ is some non-negative function.
    In particular, this implies that~\cref{eq:minimal_H} is the minimal possible solution.
\end{rmk}

\begin{example}[discrete Hamiltonian dynamics with Gaussian and Laplace momentum] \label{example:gauss_and_laplace_momentum_hmc}
Two canonical choices for $\rho$ are the standard Gaussian and the Laplace distribution, giving rise to two processes that can be simulated exactly (without discretization error) in continuous time; see \cref{sec:implementation_details} for implementation details.
In both cases, the generator is given by \cref{eq:discrete_hamiltonian_generator}
with 
$\mu$ as in \cref{eq:constant_drift} but different jump rates $Q_{x,y}(p)$. Specifically:
\begin{itemize}
    \item when $\rho$ is the standard Gaussian distribution with density $\rho(p) \propto e^{-\|p\|^2/2}$, we obtain, for $x\neq y$,
    \begin{align*}
        Q_{x, y}(p) = \Bar{Q}_{x,y} \max\left(0, \langle \sigma_{x,y}, p \rangle \right) ,
    \end{align*}
    \item when $\rho$ is the Laplace distribution with density $\rho(p) \propto e^{-\sum_{i=1}^d|p_i|}$, we obtain, for $x\neq y$,
    \begin{align*} 
        Q_{x, y}(p) 
        &= \Bar{Q}_{x,y} \max\left(0, \langle\sigma_{x,y},  \sign(p) \rangle\right) ,
    \end{align*}
    where $\sign(p) = (\sign(p_1), \sign(p_2),\dots,\sign(p_d))$ and $\Bar{Q} \in \R^{\sX \times \sX}$ is any matrix in detailed balance w.r.t. $\pi$. 
\end{itemize}

\end{example}

\subsection{Comparison with PDMPs}
\label{sec:comparison_pdmps}

In the case of full momentum refreshment, \cref{eq:discrete_HMC_generator} corresponds to the generator of a \emph{piecewise deterministic Markov process (PDMP)}~\cite{davis1984piecewise}.
An instance of the \emph{Zig-Zag process} of~\cite{bierkens2019zig} can in fact be recovered as an instance of~\cref{eq:discrete_hamiltonian_generator}: this is the case when the state space is the hypercube $\sX = \left\{ 0,1 \right\}^d$ with a uniform target $\pi = \mathcal{U}(\sX)$ and with $Q$ of the factorized form~\cref{eq:Q_factorized} with $\Bar{Q}$ being the generator of the random walk.
However, an important difference is that the roles of the discrete and continuous variables are interchanged: our objective is to sample from a (a priori non-uniform) target distribution $\pi$ on the discrete space $\sX$ that we augment with a momentum in $\R^d$ whose invariant distribution $\rho$ is easy to sample (e.g.\ Gaussian).
This has several consequences:
\begin{itemize}
    \item We can consider diffusion or full refreshment of the momentum variable ($\gamma > 0$).
    This leads to ergodicity of the reversible projection of our process and, in turn, to a simple proof of ergodicity for discrete Hamiltonian dynamics (\cref{prop:ergodicity}). In contrast, while ergodicity of the Zig-Zag process has been established~\cite{bierkens2019ergodicity,benaim2015qualitative}, it relies on much more involved arguments.
    
    \item Sharp quantitative bounds for mixing time of PDMPs, including Zig-Zag, can be established using hypocoercivity analysis~\cite{andrieu2021hypocoercivity,lu2022explicit}.
    In particular, the Zig-Zag process is a \emph{second-order lift}~\cite{eberle2024space} of a Langevin process, a property which has also been leveraged to prove accelerated convergence guarantees~\cite{eberle2025convergence}.
    In contrast, the discrete Hamiltonian dynamics correspond to a \emph{first-order lift}~\cite{chen1999lifting} of the associated reversible process generated by $\Bar{Q}$ (\cref{prop:lift}). Nonetheless, we are able to obtain exponential contraction in $L^2$ through a spectral gap analysis (\cref{prop:spectral_gap}), although we do not believe the rate to be optimal.
    
    \item From a computational point of view, discrete Hamiltonian dynamics can be simulated without resorting to thinning techniques, since the Poisson rates are piecewise linear when $\rho$ is Gaussian and piecewise constant when $\rho$ is Laplace; see \cref{alg:HMC_continuous_momentum} for details.
    In both cases, the random event times can be simulated exactly.
    This is in contrast to standard PDMPs, whose event times are often computed by thinning techniques, which can negatively affect the computational efficiency of the algorithm.
\end{itemize}

\section{Examples of discrete state spaces and orientations}

\label{sec:examples_graphs}

As described in~\cref{sec:discrete_hamiltonian_dynamics}, a natural way to construct a discrete Hamiltonian process is to rely on the knowledge of a reversible Markov generator $\Bar{Q}$ on the state space $\sX$. The set of possible transitions then needs to be provided with a notion of direction, represented by vectors $\sigma_{x,y} \in \R^d$.
While such a construction could be adapted to many settings, we give here simple generic examples that we will study in more detail in the rest of the paper.

\subsection{Reversible generators $\bar Q$}\label{sec:rev_kernels}

Let $\Bar{Q} \in \R^{\sX \times \sX}$ be the generator of a $\pi$-reversible (continuous-time) Markov process on $\sX$.
Such a generator is associated with a directed graph structure $(\mathcal{X}, E)$ which encodes the notion of local moves of the Markov process, i.e., $E \subset \sX \times \sX$ is a set of edges and $(x,y) \in E$ if and only if $\bar Q_{x,y}> 0$.
For $x \in \sX$, we denote by $N_x \coloneq \{y \colon (x,y) \in E\}$ the set of neighbours of $x$ and $n_x \coloneq |N_x|$ its size.

An important class of reversible generators consists of \emph{locally-balanced generators} \cite{zanella2020informed,power2019accelerated}.
They take the general form
\begin{align}\label{eq:Q_bar_LB}
    \bar{Q}_{x,y} & = G\left(\frac{\pi_y  \, n_x}{\pi_x \, n_y}\right) \frac{\1(y \in N_x)}{n_x}\, ,
    & (x,y) \in E \, ,
\end{align} 
where $G \colon \R^+ \to \R^+$ is a \emph{balancing function},   satisfying $G(t) = t G(1/t)$ for all $t >0$.

It can also happen that the state space $\sX$ is associated with a notion of direction, for example in case of a partial ordering of the elements~\cite{gagnon2024asymptotic}.
It is then natural to consider reversible generators that account for this supplementary information. For example, if there is a notion of one-dimensional direction encoded in the partition  $N_x = N^{-}_x \sqcup N^+_x$, one can consider \emph{stratified generators} of the form
\begin{align}\label{eq:Q_bar_LB_weighted}
\bar{Q}_{x,y} 
= 
G\left(\frac{\pi_y \, n^+_x}{\pi_x \,n^-_y}\right)\frac{\1(y \in N^+_x)}{n^+_x}
+
G\left(\frac{\pi_y \, n^-_x}{\pi_x \, n^+_y}\right)\frac{\1(y \in N^-_x)}{n^-_x}\,,
\end{align}
which are also $\pi$-reversible, where we denote $n^v_x \coloneq |N
^v_x|$ for $v\in\{-,+\}$. Similar structures have been explored e.g.\ in \cite{power2019accelerated,gagnon2024asymptotic}.

\subsection{Examples of discrete space structures}
\label{sec:specific_examples}

We give here examples of discrete spaces on which our discrete Hamiltonian dynamics can be applied.

\subsubsection{$d$-dimensional lattice}

One of the simplest examples of a discrete state space is the discretization of the Euclidean space on a standard $d$-dimensional square lattice.
For $d \geq 1$, this corresponds to
\begin{align} \label{eq:square_lattice}
    \sX_{\eps, R} \coloneq \left\{ x \in \eps \Z^d \, \colon \, \| x \|_\infty \leq R \right\} ,
\end{align}
where $\eps > 0$ is some discretization parameter and $R \in (0, +\infty]$ is some radius delimiting the size of the grid.
The set $\sX_{\eps, R}$ can be endowed with a graph structure by considering the nearest neighbours.
The set of edges is defined as
\begin{align*}
    E = \left\{ (x,y) \in \sX_{\eps, R} \times \sX_{\eps, R} \, \colon \, \| x-y \|_1 =  \eps \right\} .
\end{align*}
Each of those edges is then naturally provided with a direction which comes from the natural embedding of $\sX_{\eps, R}$ in $\R^d$.
That is, for $(x,y) \in E$,
\begin{align*}
    \sigma_{x,y} = \eps^{-1} (y-x) .
\end{align*}
We study this example in more detail in~\cref{sec:scaling_limits_infill}.
In particular, we show that, in this case, discrete Hamiltonian dynamics converge to standard Hamiltonian dynamics on $\R^d$ when the discretization parameter $\eps$ tends to $0$ (\cref{thm:infill_scaling_limit}).

\subsubsection{Discrete hypercube}\label{sec:example_discrete_hypercube}
In the hypercube, the state space is $\sX=\{0,1\}^n$.
In addition to being a convenient setting for theoretical analysis, it is also of interest for many applications which are concerned with sampling over such a state space, such as spin models in statistical physics~\cite{krauth2006statistical, newman1999monte} as well as variable selection problems in machine learning~\cite{george1993variable}.
We consider processes on the hypercube where transitions operate by adding or removing one bit at a time.

More formally, we consider the graph structure $(\sX, E)$, where 
$$
E = \{(x,y) \in \sX \times \sX \, \colon \, x_{-i} = y_{-i}, \, x_i = 1-y_i \text{ for some } i\}.
$$
In this case, for $x \in \sX$ we have $N_x = \{ y \in \sX \, \colon \, |x-y|=1 \}$, where we use the notation $|x| \coloneq \sum_{i=1}^n |x_i|$ to denote the total number of bits.
Moreover, we are in the setting discussed above where this graph structure is provided with a notion of direction, given by the total number of bits $|x|$.
Indeed, $N_x = N_x^+ \sqcup N_x^-$ with $N_x^+ \coloneq \{ y \in N_x \, \colon \, |y| = |x|+1 \}$, and similarly $N_x^- \coloneq \{ y \in N_x \, \colon \, |y| = |x|-1 \}$.

For this graph, we use one of the following definitions for $\sigma_{x,y}$:
\begin{itemize}
    \item for one-dimensional velocity $p \in \R$,
        \begin{align} \label{eq:1d_momentum_hypercube}
        \sigma_{x, y} & =
        |y| - |x| ,
        & (x,y) \in E ,
        \end{align}
    \item for multi-dimensional velocity $p \in \R^n$,
    \begin{align} \label{eq:n_momentum_hypercube}
    \sigma_{x,y} & = y - x ,
    & (x,y) \in E .
    \end{align}
\end{itemize}

We study discrete Hamiltonian dynamics on the hypercube in more detail in~\cref{sec:scaling_limit_hypercube}.
In the high-dimensional limit, we show convergence towards classical Hamiltonian dynamics (\cref{theo:scaling_limits_constrained_hypercube}).

\subsubsection{Multivariate categorical spaces}
\label{sec:example_categorical_spaces}
The hypercube can actually be seen as a particular case of the more general multivariate categorical space $\sX = [K]^n$, for an integer $K \geq 2$, where $[K] \coloneq \{1, \ldots, K\}$.
This corresponds, for example, to models with a latent categorical variable with $K$ classes for $n$ data points, as is the case in Bayesian inference on mixtures of distributions~\cite{marin2005bayesian}.
As for the hypercube case, we consider local moves that operate by changing a single coordinate at a time, that is:
\begin{align*}
    E=\{(x,y)\in\mathcal X^2 \colon x_{-i} = y_{-i}, \, x_i \ne y_i \text{ for some } i \in [n]\}.
\end{align*}
To encode directions, it is natural here to attach momentum to \emph{pairs of labels} rather than to data-point indices. Let $\Gamma = \bigl\{ \{a,b\} : 1 \le a < b \le K \bigr\}$, $d = |\Gamma| = \binom K 2$ and let $(e_{\{a,b\}})_{\{a,b\} \in \Gamma}$ be the canonical basis of $\mathbb R^d$.
If $(x, y) \in E$ and $x$ and $y$ differ only at coordinate $j \in [n]$, then 
\begin{equation}
\label{eq:sigma_categorical_space}
    \sigma_{x,y} = 
\begin{cases}
    e_{\{x_j,y_j\}}, & \hbox{ if } x_j < y_j \\
    -\,e_{\{x_j,y_j\}}, & \hbox{ if } y_j < x_j \\
\end{cases} .
\end{equation}
In this case, pairs of labels share the same momentum direction, irrespective 
of the index $j$.
This framework mirrors the recently introduced non-reversible sampler for finite mixture models \cite{ascolani2025fast}.  

\subsubsection{Permutations - ranking}
\label{sec:example_permutations_ranking}
Some applications are concerned with 
sampling probability distributions on the set of permutations of a finite set of size $n \geq 1$.
This arises for example in ranking models, with a wide range of applications~\cite{alvo2014statistical}.
In this case, the state space is $\sX = \Sigma_n := \{x \colon [n] \to [n] \colon\, x \text{ is bijective}\}$, the set of permutations of $[n] \coloneq \{1,\ldots, n\}$.

We consider transitions on $\sX$ that operate by (right) composition with transpositions of neighbouring indices.
For $i \in [n-1]$, we denote by $\tau_i \coloneq (i, i+1)$ the transposition that swaps $i$ and $i+1$ and consider the set of edges
\begin{align*}
    E = \left\{
    (x, y) \in \sX \times \sX \, \colon \, y = x \circ \tau_i, \, \hbox{for some $i \in [n-1]$}
    \right\} .
\end{align*}
These transitions can be associated to directions in $\R^n$ in the following way:
letting $(e_i)_{1 \leq i \leq n}$ be the canonical basis of $\R^n$, we consider, for $x \in \sX$ and $i \in [n-1]$,
\begin{align} \label{eq:sigma_ranking}
    \sigma_{x, x \circ \tau_i} = e_{x(i+1)} - e_{x(i)} .
\end{align}

This construction has a clear interpretation in terms of ranking.
If $\mathcal{Z} = \{ z_i \}_{1 \leq i \leq n}$ is a set of objects indexed by $i \in [n]$, then rankings of the objects in $\mathcal{Z}$ are described by permutations $x \in \Sigma_n$: the $i$-th position is occupied by $z_{x(i)}$ and transitioning from $x$ to $y = x \circ \tau_i$ means swapping the positions of $z_{x(i+1)}$ and $z_{x(i)}$.
In discrete Hamiltonian dynamics, the state space is augmented with a velocity variable $p \in \R^n$.
For weight functions $H$ as in~\cref{eq:minimal_H}, directions $\sigma_{x, x \circ \tau_i}$ as in~\cref{eq:sigma_ranking} and a Gaussian distribution $\rho$, we have
\begin{align}
    H_{x, x \circ \tau_i}(p)
    = \max\left(0, \langle \sigma_{x, x \circ \tau_i}, p \rangle \right)
    = \max(0, p_{x(i+1)} - p_{x(i)}).
\end{align}
Hence $p \in \R^n$ can be interpreted as a vector of scores of the objects in $\mathcal{Z}$: transitions from $x$ to $x \circ \tau_i$ are possible only when $p_{x(i+1)} > p_{x(i)}$.
This structure can be of interest in models where the relative ranking of two objects is independent of the other objects, as for the Plackett-Luce model~\cite{plackett1975analysis,luce1959individual}.

\section{Theoretical properties} \label{sec:theory}

We investigate here theoretical properties of the discrete Hamiltonian processes defined in \cref{sec:discrete_hamiltonian_dynamics}.
In particular, we are interested in theoretical advantages of non-reversible algorithms, as opposed to reversible ones, in the task of sampling the target distribution $\pi$.

We consider here a Markov process $Z_t = (X_t, P_t)$ on $\sX \times \R^d$ with generator $\Hat{\sL} = \Hat{\sL}^H + \gamma \Hat{\sL}^D_p$ as in~\cref{eq:discrete_HMC_generator}, with $\sL^D_p$ the generator of a $\rho$-reversible Markov semigroup.
We assume that the assumptions of~\cref{thm:invariance_solution} hold and the generator $Q$ and the drift $\mu$ are such that~\cref{eq:invariance_condition} holds, i.e., $(Z_t)_{t \geq 0}$ is $\Hat{\pi}$-invariant.
This in particular includes (but is not restricted to) the case where $Q$ is of the form in~\cref{eq:Q_factorized} with weight functions satisfying~\cref{eq:constant_drift_condition}.

\subsection{Adjoint and reversible part of the generator}

In this section, we will consider a $\Hat{\pi}$-reversible Markov process $(Z^{rev}_t)_{t \geq 0}$ which is the reversible projection of $(Z_t)_{t \geq 0}$.
This process is generated by the symmetric part of the generator $\Hat{\sL}$.

\begin{prop}
    Let $\Hat{\sL}^*$ be the $L^2(\Hat{\pi})$-adjoint of $\Hat{\sL}$.
    Then $\Hat{\sL}^*$ is given for test functions $f \in \sC^\infty_c(\sX \times \R^d)$ by:
    \begin{align} \label{eq:adjoint_generator}
        \Hat{\sL}^* f(x,p) = \sum_{y \in \sX} Q^\dagger_{x,y}(p) f(y,p) - \langle \mu_x(p), \nabla_p f(x,p) \rangle + \gamma \Hat{\sL}^D_p f(x,p) ,
    \end{align}
    where, for $p \in \R^d$, we define $Q^\dagger(p) \coloneq \diag(1/\pi) Q(p)^\top \diag(\pi) - \diag(q(p)) \in \R^{\sX \times \sX}$,
    with $q(p) \in \R^\sX$ defined in~\cref{eq:q}.
    The symmetric part $\Hat{\sL}^{rev} \coloneq \frac{1}{2}(\Hat{\sL}+\Hat{\sL}^*)$ of $\Hat{\sL}$ is then given for $f \in \sC^\infty_c(\sX \times \R^d)$ by:
    \begin{align} \label{eq:reversible_generator}
        \Hat{\sL}^{rev} f(x,p)  = \sum_{y \in \sX} Q^{rev}_{x,y}(p) f(y,p) + \gamma \Hat{\sL}^D_p f(x,p) ,
    \end{align}
    where, for $p \in \R^d$, we define $Q^{rev}(p) \coloneq \frac{1}{2} \left( Q(p) + Q^\dagger(p) \right) \in \R^{\sX \times \sX}$.
\end{prop}

\begin{proof}
    Recall that $\sC^\infty_c(\sX \times \R^d)$ is a core for $\Hat{\sL}$ by~\cref{ass:core}.
    Hence, to prove the result it suffices to show that, for every $f,g \in \sC^\infty_c(\sX \times \R^d)$,
    $\langle \Hat{\sL} f, g \rangle_{L^2(\Hat{\pi})} = \langle f, \Hat{\sL}^* g \rangle_{L^2(\Hat{\pi})}$,
    where $\Hat{\sL}^*$ is defined by~\cref{eq:adjoint_generator}. \Cref{eq:reversible_generator} follows by writing $\Hat{\sL}^{rev} = \frac{1}{2}(\Hat{\sL}+\Hat{\sL}^*)$.
    Also, since $\Hat{\sL}^D_p$ is self-adjoint by assumption, we can consider here w.l.o.g. that $\gamma = 0$ and $\Hat{\sL} = \Hat{\sL}^H$.
    
    Let us consider test functions $f,g \in \sC^\infty_c(\sX \times \R^d)$.
    By the definition of $\Hat{\sL}$,
    \begin{align*}
        \langle \Hat{\sL} f, g \rangle_{L^2(\Hat{\pi})}
        = & \int_{\R^d} \sum_{x \in \sX} \left( \sum_{y \in \sX} Q_{x,y}(p) f(y,p) \right) g(x,p) \pi_x \d \rho(p)
        + \sum_{x \in \sX} \int_{\R^d} \langle \mu_x(p) , \nabla_p f(x,p) \rangle g(x,p) \d \rho(p) \pi_x \\
        = & \int_{\R^d} \sum_{x \in \sX} \left( \sum_{y \in \sX} Q_{y,x}(p) \frac{\pi_y}{\pi_x} g(y,p) \right) f(x,p) \pi_x \d \rho(p) \\
        & - \sum_{x \in \sX}  \int_{\R^d} f(x,p) g(x,p) q_x(p) \d \rho(p) \pi_x
        - \sum_{x \in \sX} \int_{\R^d} \langle \mu_x(p) , \nabla_p g(x,p) \rangle f(x,p) \d \rho(p) \pi_x \\
        = & \int_{\R^d} \sum_{x \in \sX} \left( \sum_{y \in \sX} Q_{x,y}^\dagger(p) g(y,p) \right) f(x,p) \pi_x \d \rho(p)
        - \sum_{x \in \sX} \int_{\R^d} \langle \mu_x(p) , \nabla_p g(x,p) \rangle f(x,p) \d \rho(p) \pi_x ,
    \end{align*}
    where we used the invariance condition~\cref{eq:invariance_condition} in the second line.
    This is the desired result.
\end{proof}

Note that, for each fixed $p \in \R^d$, $Q^{rev}(p)$ is, by construction, a $\pi$-reversible generator and hence the generator of a $\pi$-reversible Markov process on $\sX$.
In particular, if $Q$ is of the form~\cref{eq:Q_factorized}, then we have for $p \in \R^d$ and $y \neq x$:
\begin{align*}
    Q^{rev}_{x,y}(p) = \frac{1}{2} \Bar{Q}_{x,y} \left[ H_{x,y}(p) + H_{y,x}(p) \right]. 
\end{align*}

Throughout this section, in addition to~\cref{ass:core}, we will make the following technical assumption.
\begin{asmp}
    The operator $\Hat{\sL}^{rev}$ with domain $\Hat{\sD}^{rev} \subset L^2(\Hat{\pi})$ is the self-adjoint operator associated to a closed symmetric Dirichlet form.
    Moreover, $\sC^\infty_c(\sX \times \R^d)$ is a core for this operator.
    In particular, $\Hat{\sD} \subset \Hat{\sD}^{rev}$ and $\langle f, \Hat{\sL} g \rangle_{L^2(\Hat{\pi})} = \langle f, \Hat{\sL}^{rev} g \rangle_{L^2(\Hat{\pi})}$ for every $f,g \in \Hat{\sD}$.
\end{asmp}
Under this assumption, $\Hat{\sL}^{rev}$ is self-adjoint and is by construction the generator of a $\Hat{\pi}$-reversible Markov process on \hbox{$\sX \times \R^d$} which we will denote by $Z^{rev}_t = (X^{rev}_t, P^{rev}_t)$.

\subsection{Ergodicity}

We start by studying the ergodicity of the Markov processes generated by $\Hat{\sL}$ and $\Hat{\sL}^{rev}$.
Ergodicity of $(Z_t)_{t \geq 0}$ (and similarly of $(Z^{rev}_t)_{t \geq 0}$) can be defined as the almost sure convergence of ergodic averages, that is, for every $f \in L^1(\Hat{\pi})$, it holds
$$
    \lim_{T \to +\infty} \frac{1}{T} \int_0^T f(Z_t) \d t = \E_{\Hat{\pi}} f.
$$
We refer to~\cite[][Section~1.5]{walters2000introduction} or~\cite[][Chapter~9]{kallenberg1997foundations} for theoretical background and equivalent definitions of ergodicity.
The following result gives sufficient conditions for both the non-reversible process $(Z_t)_{t \geq 0}$ and the reversible process $(Z_t^{rev})_{t \geq 0}$ to be ergodic.

\begin{prop} \label{prop:ergodicity}
    Assume $\gamma \sL^D_p$ is the generator of a $\rho$-reversible ergodic Markov process on $\R^d$, and that $Q^{rev}(p)$ is irreducible on a set of positive $\rho$-measure.
    Then the processes $(Z_t)_{t \geq 0}$ and $(Z^{rev}_t)_{t \geq 0}$ are ergodic.

    In particular, the result holds for $Q$ as in~\cref{eq:Q_factorized} if $\Bar{Q}$ is irreducible, the weight functions $H_{x,y}$ satisfy~\cref{eq:constant_drift_condition} and 
    \begin{align*}
        \rho \left(
        \left\{
        p \in \R^d \, \colon \,
        \inf_{x, y \in \sX, \Bar{Q}_{x,y} > 0} \left| \langle \sigma_{x,y}, \nabla \log \rho(p) \rangle \right| > 0
        \right\}
        \right)
        > 0.
    \end{align*}
\end{prop}

\begin{proof}
    As discussed in~\cite[Proposition 2.1]{bhattacharya1982functional}, it suffices to check that the nullspace of $\Hat{\sL}$ (resp. $\Hat{\sL}^{rev}$) is the one-dimensional space spanned by constant functions. Since $\langle \Hat{\sL}f, f \rangle_{L^2(\Hat{\pi})} = \langle \Hat{\sL}^{rev} f, f \rangle_{L^2(\Hat{\pi})}$ for every $f \in \Hat{\sD}$, it suffices to check this property for the reversible part $\Hat{\sL}^{rev}$.

    Now assume $f \in \Hat{\sD}^{rev}$ is s.t. $\langle \Hat{\sL}^{rev} f, f \rangle_{L^2(\Hat{\pi})} = 0$.
    Then, by definition,
    \begin{align*}
        0 & = \int_{\R^d} \left( \sum_{x,y \in \sX} Q^{rev}_{x,y}(p) f(x,p) f(y,p) \pi_x \right) \d \rho(p) + \gamma \sum_{x \in \sX} \langle \sL^D_p f(x, \cdot) , f(x, \cdot) \rangle_{L^2(\rho)} \pi_x .
    \end{align*}
    This is the sum of two non-positive terms, so both of them are $0$.
    In the second term, ergodicity of $\sL^D_p$ implies that $f$ is equal in $L^2(\Hat{\pi})$ to a function independent of $p$.
    Then, in the first term, the assumption that $Q^{rev}$ is irreducible on a set of positive $\rho$-measure implies that $f$ is also independent of $x$.
    
    For the particular case of $Q$ as in~\cref{eq:Q_factorized}, we have for $x,y \in \sX$ with $\Bar{Q}_{x,y} > 0$ and $p \in \R^d$:
    $$
    Q^{rev}_{x,y}(p) = \frac{1}{2} \Bar{Q}_{x,y} \left[ H_{x,y}(p) + H_{y,x}(p) \right]
    \geq \frac{1}{2} \Bar{Q}_{x,y} \left| \langle \sigma_{x,y}, \nabla \log \rho(p) \rangle \right|  .
    $$
    If $\Bar{Q}$ is irreducible, this implies irreducibility of $Q^{rev}(p)$ when \hbox{$\inf_{\Bar{Q}_{x,y} > 0} | \langle \sigma_{x,y}, \nabla \log \rho(p) \rangle | > 0$}.
\end{proof}

\subsection{Spectral gap and asymptotic variance}

We compare here the non-reversible dynamics generated by~\cref{eq:discrete_HMC_generator} to their reversible counterparts in terms of sampling performance.
While several criteria are available, we focus here on the asymptotic variance.
In the following, we let $L^2_0(\Hat{\pi})$ be the space of functions $f \in L^2(\Hat{\pi})$ with zero mean.

\begin{defi}[Asymptotic variance]
    Let $(Z_t)_{t \geq 0}$ be an ergodic Markov process on $\sX \times \R^d$ and invariant distribution $\Hat{\pi} \in \sP(\sX \times \R^d)$ and generator $\Hat{\sL}$, started in stationarity.
    Then for any $f \in \Rg(\Hat{\sL})$ we define the asymptotic variance by
    \begin{align} \label{eq:asymptotic_variance}
        \var(f, \Hat{\sL}) \coloneq 2 \langle f, g \rangle_{L^2(\Hat{\pi})},
    \end{align}
    where $g \in L^2_0(\Hat{\pi})$ solves the \emph{Poisson equation} $- \Hat{\sL} g = f$.

    In particular, it follows from the central limit theorem for continuous-time Markov processes that the stochastic process $\left\{ n^{-1/2} \int_0^{n t} f(Z_s) \d s, \, t \geq 0 \right\}$ converges weakly to a Wiener process with zero drift and variance parameter $\var(f, \Hat{\sL})$~\cite[Theorem 2.1]{bhattacharya1982functional}.
\end{defi}

\subsubsection{Spectral gap analysis}

As shown in \cref{eq:asymptotic_variance}, the asymptotic variance strongly relies on the well-posedness of the Poisson equation $- \Hat{\sL} g = f$.
A way to ensure existence and uniqueness of a solution $g \in L^2(\Hat{\pi})$ is to establish the exponential contractivity of the associated Markov semigroup.
We study here the \emph{spectral gap} associated with the reversible generator $\Hat{\sL}^{rev}$, which we recall is defined as:
\begin{align*}
    \gap (\Hat{\sL}^{rev}) \coloneq \inf
    \left\{
         - \frac{\langle \Hat{\sL}^{rev} f, f \rangle_{L^2(\Hat{\pi})}}{\| f \|^2_{L^2(\Hat{\pi})}} \; \colon \;  f \in \Hat{\sD}^{rev} \setminus \{ 0 \}, \, \, \E_{\Hat{\pi}} f = 0 
    \right\} .
\end{align*}
The following result shows that the self-adjoint operator $\Hat{\sL}^{rev}$ has a positive spectral gap.

\begin{prop} \label{prop:spectral_gap}
    Assume $\lambda_p \coloneq \gap(\sL^D_p) > 0$. Denote $\Bar{\lambda}_x \coloneq \gap( \E_\rho Q^{rev}(p))$ and, for every $p \in \R^d$, $\lambda_x(p) \coloneq \gap (Q^{rev}(p))$.
    Then the self-adjoint operator $\Hat{\sL}^{rev}$ has a spectral gap satisfying:
    \begin{align} \label{eq:gap_bound}
        \min \left\{ \Bar{\lambda}_x, \gamma \lambda_p \right\}
        \geq \gap(\Hat{\sL}^{rev}) 
        \geq
        \lambda^*,
    \end{align}
    where $\lambda^* \in [0, \gamma \lambda_p ]$ solves
    \begin{align*}
            \int_{\R^d} \frac{\gamma \lambda_p}{\lambda_x(p) + \gamma \lambda_p - \lambda^*} \d \rho(p) = 1
    \end{align*}
    when $\int \frac{\gamma \lambda_p}{\lambda_x} \d \rho > 1$, and $\lambda^* = \gamma \lambda_p$ otherwise.
\end{prop}

\begin{rmk}
    Let us make some comments on~\cref{prop:spectral_gap}.
    First, note that using Jensen's inequality, we have for every $\lambda \in [0, \gamma \lambda_p ]$,
    \begin{align*}
        \int_{\R^d} \frac{\gamma \lambda_p}{\lambda_x(p) + \gamma \lambda_p - \lambda} \d \rho(p) \geq \frac{\gamma \lambda_p}{ \E_\rho \lambda_x + \gamma \lambda_p - \lambda}.
    \end{align*}
   Thus, setting the r.h.s.\ equal to $1$, we obtain $\lambda^* \leq \E_\rho \lambda_x \leq \Bar{\lambda}_x$.
    In particular, all those inequalities become equalities in the case where $Q^{rev}$ is independent of $p$, in which case we recover:
    \begin{align*}
        \gap(\Hat{\sL}^{rev}) =
        \min \left\{ \gap (Q^{rev}), \gamma \gap (\sL^D_p) \right\}  .
    \end{align*}
    More generally, it follows from the proof of~\cref{prop:spectral_gap} that the lower bound in~\cref{eq:gap_bound} is sharp when $Q^{rev}(p)  = \Bar{Q} \lambda_x(p)$ for some $\pi$-reversible generator $\Bar{Q}$.
    
    Also observe that $\lambda^*$ is a non-decreasing function of $\gamma$.
    In the regime $\gamma \to +\infty$, we have
    \begin{align*}
        \int_{\R^d} \frac{\gamma \lambda_p}{\lambda_x(p) + \gamma \lambda_p - \lambda} \d \rho(p) \simeq 1 - \frac{\E_\rho \lambda_x - \lambda}{\gamma \lambda_p} .
    \end{align*}
    Setting the r.h.s.\ equal to $1$ we obtain (at least informally) $\liminf \lambda^* \geq \E_\rho \lambda_x$.

    Finally observe that, by~\cref{prop:spectral_gap}, the reversible part of the generator $\Hat{\sL}^{rev}$ has a positive spectral gap as soon as $\rho (\lambda_x > 0) > 0$.
    For $\lambda, \eps \in [0, \gamma \lambda_p]$,
    \begin{align*}
        \int_{\R^d} \frac{\gamma \lambda_p}{\lambda_x(p) + \gamma \lambda_p - \lambda} \d \rho(p) \leq
        \frac{\gamma \lambda_p}{\gamma \lambda_p - \lambda} (1-P_\eps) + \frac{\gamma \lambda_p}{\eps + \gamma \lambda_p - \lambda} P_\eps,
    \end{align*}
    with $P_\eps = \rho(\lambda_x \geq \eps)$.
    As before, setting the r.h.s.\ equal to $1$ gives:
    \begin{align*}
        \lambda^* \geq \gamma \lambda_p - \frac{1}{2} \left[
        (\gamma \lambda_p - \eps) + \sqrt{(\gamma \lambda_p - \eps)^2 + 4 \gamma \lambda_p \eps (1-P_\eps)}
        \right]
        \geq \frac{1}{2} \eps P_\eps .
    \end{align*}
    In contrast, if $\lambda_x = 0$ for $\rho$-a.e. $p \in \R^d$, we of course recover $\lambda^* = 0$.
\end{rmk}

\begin{proof}
    The upper bound on the spectral gap follows from standard arguments and we will focus on proving the lower bound.
    We will leverage the orthogonal decomposition $L^2_0(\Hat{\pi}) = E_p \bigoplus G \bigoplus E_x$, where $E_p$, $E_x$ are the respective ranges of the orthogonal projectors:
    \begin{align*}
        \Pi_p f \coloneq \int_{\R^d} f \d \rho, \quad
        \Pi_x f \coloneq \sum_{x \in \sX} f \pi_x,
    \end{align*}
    and $G = \Rg(I-\Pi_p-\Pi_x)$.
    With this notation, let us consider $f \in \Hat{\sD}^{rev} \cap L^2_0(\Hat{\pi})$ that we decompose accordingly as $f = \Bar{f}_p + g + \Bar{f}_x$.
    Then, using the definition of $\Hat{\sL}^{rev}$ in~\cref{eq:reversible_generator},
    \begin{align*}
        - \langle \Hat{\sL}^{rev} f , f \rangle_{L^2(\Hat{\pi})}
        & = - \int_{\R^d} \langle Q^{rev}(p) f(\cdot,p),  f(\cdot,p) \rangle_{L^2(\pi)} \d \rho(p)
        - \gamma \sum_{z \in \sX} \langle \sL^D_p f(z,\cdot) , f(z,\cdot) \rangle_{L^2(\rho)} \pi_z \\
        & \geq \int_{\R^d} \lambda_x(p) \sum_{z \in \sX}  (f-\Bar{f}_x)^2(z,p) \pi_z \d \rho(p) + \gamma \lambda_p \sum_{z \in \sX} \int_{\R^d} (f - \Bar{f}_p)^2(z,p) \d \rho(p) \pi_z \\
        & = \sum_{z \in \sX} \int_{\R^d} \lambda_x(p) (\Bar{f}_p + g)^2(z,p) \d \rho(p) \pi_z + \gamma \lambda_p \sum_{z \in \sX} \int_{\R^d} (g + \Bar{f}_x)^2(z,p) \d \rho(p) \pi_z \\
        & = \sum_{z \in \sX}
        \left[
        \int_{\R^d} \lambda_x(p) (\Bar{f}_p + g)^2(z,p) \d \rho(p) + \gamma \lambda_p \int_{\R^d} g^2(z,p) \d \rho(p)
        \right] \pi_z
        + \gamma \lambda_p \| \Bar{f}_x \|^2_{L^2(\Hat{\pi})} .
    \end{align*}
    Observe that, letting $h = \Bar{f}_p + g$, we have $\Bar{f}_p = \Pi_p h$, and, for every $z \in \sX$,
    \begin{multline*}
        \int_{\R^d} \lambda_x(p) (\Bar{f}_p + g)^2(z,p) \d \rho(p) + \gamma \lambda_p \int_{\R^d} g^2(z,p) \d \rho(p) \\
        = \int_{\R^d} \lambda_x(p) h^2(z,p) \d \rho(p) + \gamma \lambda_p \int_{\R^d} (h - \Pi_p h)^2(z,p) \d \rho(p) \\
        \geq \gamma \lambda_p \lambda^* \left( \int_{\R^d} (\Bar{f}_p + g)^2(z,p) \d \rho(p) \right),
    \end{multline*}
    where the constant $\lambda^*$ is defined as the value of the minimization problem:
    \begin{align} \label{eq:lambda_star}
        \lambda^* \coloneq \inf_{\substack{h \in L^2(\rho) \\ \| h \|_{L^2(\rho)} = 1}} \sE(h) ,
    \end{align}
    with the functional $\sE$ defined for $h \in L^2(\rho)$ by
    \begin{align*}
        \sE(h) & \coloneq \int_{\R^d} h(p)^2 \phi(p) \d \rho(p) + \| h \|^2_{L^2(\rho)} - \left( \int_{\R^d} h \d \rho \right)^2, 
    \end{align*}
    and $\phi : p \in \R^d \mapsto (\gamma \lambda_p)^{-1} \lambda_x(p)$.
    Plugging this into the previous inequality and using the orthogonal decomposition of $f$ leads to:
    \begin{align*}
        - \langle \Hat{\sL}^{rev} f , f \rangle_{L^2(\Hat{\pi})} & \geq \gamma \lambda_p \lambda^* \sum_{z \in \sX} \int_{\R^d} (\Bar{f}_p + g)^2(z,p) \d \rho(p) \pi_z + \gamma \lambda_p \| \Bar{f}_x \|^2_{L^2(\Hat{\pi})} \\
        & = \gamma \lambda_p \lambda^* \| \Bar{f}_p + g \|^2_{L^2(\Hat{\pi})} + \gamma \lambda_p \| \Bar{f}_x \|^2_{L^2(\Hat{\pi})} \\
        & \geq \gamma \lambda_p \min(1, \lambda^*) \| f \|^2_{L^2(\Hat{\pi})}.
    \end{align*}
    Let us now compute the value of $\lambda^*$.
    Observe that, for $h \in L^2(\rho)$, $\sE(h) = \langle h, \sA h \rangle_{L^2(\rho)}$ where $\sA$ is the non-negative self-adjoint operator defined by $\sA = M_\phi + \Id - \Pi$ with $M_\phi$ the multiplication by $\phi$ and $\Pi$ the projection on the space of constant functions.
    In particular, by the min-max principle for self-adjoint operators~\cite[][Theorem~XIII.2]{reed1972methods}, $\lambda^*$ as defined in~\cref{eq:lambda_star} is either the bottom of the essential spectrum or an eigenvalue of $\sA$.
    Since $\sA$ is here a rank-one perturbation of the self-adjoint operator $\sB = M_\phi + \Id$, using that $\phi$ is non-negative, we have that $\sigma_{ess}(\sA) = \sigma_{ess}(\sB) \subset [1, +\infty)$~\cite[][Section~XIII.4]{reed1972methods}.
    On the other hand, the existence of a discrete eigenvalue $\lambda \in [0, 1)$ for $\sA$ is equivalent to the existence of a solution $\lambda \in [0, 1)$ to the equation
    \begin{align} \label{eq:eigenvalue_problem}
        \psi(\lambda) = \int_{\R^d} \frac{\d \rho(p)}{\phi(p) + 1 - \lambda} = 1 ,
    \end{align}
    which is associated to the eigenfunction $f \propto 1/(\phi+1-\lambda)$.
    Observe that the function $\lambda \mapsto \psi(\lambda)$ is increasing with $\psi(0) \leq 1$ and $\lim_{\lambda \to 1^-} \psi(\lambda) = \int \frac{\d \rho}{\phi} \in (0, +\infty]$.
    Hence, there exists a discrete eigenvalue $\lambda^* \in [0,1)$ whenever $\int \frac{\d \rho}{\phi} > 1$ and $\lambda^*$ is defined by~\cref{eq:eigenvalue_problem}.
    Otherwise, we have $\lambda^* = \inf \sigma_{ess}(\sB) \geq 1$ and this gives the desired result.
\end{proof}
\subsubsection{Asymptotic variance ordering}
\label{sec:comparison_reversible}

We now compare the performance in terms of asymptotic variance of the discrete Hamiltonian process $(Z_t)_{t \geq 0}$,  its reversible projection $(Z^{rev}_t)_{t \geq 0}$ and the base reversible process, which is the continuous-time $\pi$-reversible Markov process on $\sX$ with generator $\E_\rho Q^{rev}(p)$.
In the examples considered in this paper with $Q$ of the form~\cref{eq:Q_factorized}, the base reversible process can be compared with the reversible process with generator $\Bar Q$.
In particular, in the case of a Gaussian momentum distribution $\rho$, the two generators coincide up to a time rescaling factor.
The following result shows that the non-reversible process always outperforms its reversible counterparts in terms of asymptotic variance.
Moreover, in the regime where $\gamma \gg 1$, the non-reversible process also improves on the base reversible process on $\sX$.
\begin{prop} \label{prop:asymptotic_variance}
    Assume $\Hat{\sL}^{rev}$ has a positive spectral gap.
    Then, for any $f \in L^2_0(\Hat{\pi})$,
    \begin{align} \label{eq:variance_comparison1}
        \var(f, \Hat{\sL}) \leq \var(f, \Hat{\sL}^{rev})\,.
    \end{align}
    Moreover, for any function $\Tilde{f} \in L^2_0(\pi)$,
    \begin{align} \label{eq:variance_comparison2}
         \var( \Tilde{f} \circ \tau, \Hat{\sL})
         \leq \var( \Tilde{f} \circ \tau, \Hat{\sL}^{rev})
         = \var(\Tilde{f}, \E_\rho Q^{rev}(p)) + \sO(\gamma^{-1}),
    \end{align} 
    where $\tau : \sX \times \R^d \to \sX$ is the projection on the first component. In particular, for $Q$ as in~\cref{eq:Q_factorized} with $H_{x,y}$ satisfying~\cref{eq:constant_drift_condition}, we have
    \begin{align*}
        \var( \Tilde{f} \circ \tau, \Hat{\sL})
         \leq  2 C^{-1} \var(\Tilde{f}, \Bar{Q} ) + \sO(\gamma^{-1}) ,
    \end{align*}
    where $C = \min_{\Bar{Q}_{x,y} > 0} \E_\rho | \langle \sigma_{x,y}, \nabla \log \rho(p) \rangle |$ .
\end{prop}
\begin{proof}
    We start by proving~\cref{eq:variance_comparison1}.
    Since $\lambda = \gap(\Hat{\sL}^{rev}) > 0$, $\Hat{\sL}^{rev}$ is bijective from $\Hat{\sD}^{rev} \cap L^2_0(\Hat{\pi})$ onto $L^2_0(\Hat{\pi})$, with bounded inverse.
    By symmetry, this also implies, for any $f \in L^2_0(\Hat{\pi})$ and any $t > 0$,
    \begin{align*}
        \left\| e^{t \Hat{\sL}} f \right\|_{L^2(\Hat{\pi})} \leq e^{-\lambda t} \left\| f \right\|_{L^2(\Hat{\pi})} ,
    \end{align*}
    with $(e^{t \Hat{\sL}})_{t \geq 0}$ the Markov semigroup associated to $(Z_t)_{t \geq 0}$.
    Thus defining $g = \int_0^\infty e^{s \Hat{\sL}} f \d s$, we have $g \in \Hat{\sD} \cap L^2_0(\Hat{\pi})$ and it is a solution to the Poisson equation $- \Hat{\sL} g = f$.
    Then using that $\Hat{\sD} \subset \Hat{\sD}^{rev}$ and the Cauchy-Schwarz inequality gives
    \begin{align*}
        \langle f, g \rangle^2_{L^2(\Hat{\pi})}
        \leq \langle f, (-\Hat{\sL}^{rev})^{-1} f \rangle_{L^2(\Hat{\pi})} 
        \langle g, (-\Hat{\sL}^{rev}) g \rangle_{L^2(\Hat{\pi})}
        \leq \langle f, (-\Hat{\sL}^{rev})^{-1} f \rangle_{L^2(\Hat{\pi})} 
        \langle f, g \rangle_{L^2(\Hat{\pi})} ,
    \end{align*}
    which is the desired result.

    We now turn to the proof of~\cref{eq:variance_comparison2}.
    The particular case of $Q$ of the form~\cref{eq:Q_factorized} follows by observing that in this case, for $x \neq y$, $\E_\rho Q^{rev}_{x,y}(p) \geq \frac{C}{2} \Bar{Q}_{x,y}$.
    
    Let us denote $\Tilde{Q} = \E_\rho Q^{rev}(p) \in \R^{\sX \times \sX}$, which by assumption is a $\pi$-reversible generator on $\sX$ with positive spectral gap (cf.~\cref{eq:gap_bound}).
    Let us denote by $\Hat{\sL}^Q$ the operator defined for $f \in L^2(\Hat{\pi})$ by:
    \begin{align*}
        \Hat{\sL}^Q f(x,p) = \sum_{y \in \sX} Q^{rev}_{x,y}(p) f(y,p) ,
    \end{align*}
    such that by construction $\Hat{\sL}^{rev} = \Hat{\sL}^Q + \gamma \Hat{\sL}^D_p$ with $\Hat{\sL}^D_p = \Id \otimes \sL^D_p$.
    We also denote by $\Tilde{\sL} = \Tilde{Q} \otimes \Id$ the average of the operator $\Hat{\sL}^Q$ w.r.t. the variable $p$.
    Let $\Tilde{f} \in L^2_0(\pi)$.
    Then we can find $\Tilde{g} \in L^2_0(\pi)$ solving the Poisson equation $- \Tilde{Q} \Tilde{g} = \Tilde{f}$.
    After lifting to $\sX \times \R^d$, this becomes, in our notations, $- \Tilde{\sL} (\Tilde{g} \circ \tau) = \Tilde{f} \circ \tau$.
    Let $g \in \Hat{\sD}^{rev} \cap L^2_0(\Hat{\pi})$ be the solution to the Poisson equation $- \Hat{\sL}^{rev} g = \Tilde{f} \circ \tau$.
    Introducing $h = \gamma (g - \Tilde{g} \circ \tau) \in \Hat{\sD} \cap L^2_0(\Hat{\pi})$ and using that $\Hat{\sL}^D_p (\Tilde{g} \circ \tau) = 0$, the above equation reads:
    $$
    (\Hat{\sL}^D_p + \gamma^{-1} \Hat{\sL}^Q) h = - \Tilde{f} \circ \tau - \Hat{\sL}^Q (\Tilde{g} \circ \tau) = (\Tilde{\sL} - \Hat{\sL}^Q) (\Tilde{g} \circ \tau) .
    $$
    Now, taking expectations w.r.t. $p$, we have by construction  $\E_\rho (\Tilde{\sL} - \Hat{\sL}^Q) (\Tilde{g} \circ \tau) = (\Tilde{Q} - \E_\rho \Hat{\sL}^Q) \Tilde{g} = 0$.
    Hence $(\Tilde{\sL} - \Hat{\sL}^Q) (\Tilde{g} \circ \tau) \in \Rg(\Hat{\sL}^D_p)$ and we can find $\Tilde{h} \in L^2_0(\Hat{\pi})$ s.t.
    $$
    \Hat{\sL}^D_p \Tilde{h} = (\Tilde{\sL} - \Hat{\sL}^Q) (\Tilde{g} \circ \tau) .
    $$
    Introducing $r = h - \Tilde{h}$, the previous equation can be simplified further into:
    $$
    (\Hat{\sL}^Q + \gamma \Hat{\sL}^D_p) r = -\Hat{\sL}^Q \Tilde{h} .
    $$
    Using that $\Tilde{h}$ is independent of $\gamma > 0$ and that $\gap(\Hat{\sL}^Q + \gamma \Hat{\sL}^D_p)$ is increasing with $\gamma$ (this follows from the proof of \cref{prop:spectral_gap}) we have $\| r \|_{L^2(\Hat{\pi})} = \sO(1)$ as $\gamma \to +\infty$.
    Finally, since $g = \Tilde{g} \circ \tau + \gamma^{-1} (\Tilde{h}+r)$, we have by definition of the asymptotic variance:
    \begin{align*}
        \var(\Tilde{f} \circ \tau, \Hat{\sL}^{rev})
        = 2 \langle \Tilde{f}, \Tilde{g} \rangle_{L^2(\pi)} + 2 \gamma^{-1} \langle \Tilde{f} \circ \tau, \Tilde{h} + r \rangle_{L^2(\Hat{\pi})} = \var(\Tilde{f}, \Tilde{Q}) + \sO(\gamma^{-1}) . 
    \end{align*}
\end{proof}

\subsection{Relaxation time}

We finish this section by giving a lower bound on the relaxation time in $L^2(\Hat{\pi})$ of the discrete Hamiltonian process $(Z_t)_{t \geq 0}$.
We consider the following definition of non-asymptotic relaxation time, taken from~\cite[][Definition~5]{eberle2024non}.

\begin{defi}
    Let $\sL$ be the generator of a Markov process $(Z_t)_{t \geq 0}$ on a space $\mathcal{Z}$ that is invariant w.r.t. a probability distribution $\pi \in \sP(\mathcal{Z})$.
    The \emph{relaxation time} of $\sL$ is defined as:
    \begin{align}
        t_{rel}(\sL) \coloneq \inf \left\{ t \geq 0 \ \colon \ \left\| e^{t \sL} f \right\|_{L^2(\pi)} \leq e^{-1} \left\| f \right\|_{L^2(\pi)} \ \text{for all $f \in L^2_0(\pi)$} \right\} ,
    \end{align}
    where $(e^{t \sL})_{t \geq 0}$ denotes the Markov semigroup associated to $(Z_t)_{t \geq 0}$.
\end{defi}
The following proposition gives lower bounds on the relaxation times of the discrete Hamiltonian process $(Z_t)_{t \geq 0}$ in terms of the relaxation time of the base reversible process on~$\sX$.
\begin{prop} \label{prop:relaxation_time}
    Assume $\Hat{\sL}^{rev}$ has a positive spectral gap and, for $p \in \R^d$, define $R(p) \coloneq \max_{x \in \sX} \sum_{y \neq x} Q_{x,y}(p)$.
    Then,
    \begin{align*}
        t_{rel}(\Hat{\sL})
        \geq \frac{1}{2 \sqrt{2}} \sqrt{t_{rel}(\E_\rho [ R(p) Q^{rev}(p) ] )} .
    \end{align*}
    In particular, for $Q$ of the form~\cref{eq:Q_factorized} with $H_{x,y}$ of the form~\cref{eq:minimal_H} and $\| \sigma_{x,y} \| = 1$, defining $\Bar{R} \coloneq \max_{x \in \sX} \sum_{y \neq x} \Bar{Q}_{x,y}$, we have
    \begin{align*}
        t_{rel}(\Hat{\sL})
        \geq \frac{1}{2 \sqrt{2 \Bar{R} \E_\rho [ \| \nabla \log \rho(p) \|^2 ]}} \sqrt{t_{rel}(\Bar{Q})} .
    \end{align*}
\end{prop}
\begin{proof}
    The proof proceeds by estimating the \emph{singular value gap}, defined as
    \begin{align*}
        \mathrm{s}(\Hat{\sL}) \coloneq \inf
        \left\{
            \frac{\left\| \Hat{\sL} f \right\|_{L^2(\Hat{\pi})}}{\| f \|_{L^2(\Hat{\pi})}} \, \colon \, f \in \Hat{\sD} \cap L^2_0(\Hat{\pi}) \setminus \{0\}
        \right\} ,
    \end{align*}
    and applying~\cite[][Lemma~10]{eberle2024non} (see also~\cite{chatterjee2025spectral}).
    For a test function $f \in L^2_0(\pi)$, denoting \hbox{$\tau : \sX \times \R^d \to \sX$} the projection on the first component, we have
    \begin{align*}
        \left\| \Hat{\sL} (f \circ \tau) \right\|^2_{L^2(\Hat{\pi})} & = \langle f, \E_\rho [ Q(p)^* Q(p)] f \rangle_{L^2(\pi)} ,
    \end{align*}
    where $Q(p)^*$ is the $L^2(\pi)$-adjoint of $Q(p)$.
    For every $p \in \R^d$, defining the Markov transition kernel $K(p) = R(p)^{-1} Q(p) + \Id$, we have
    \begin{align*}
        R(p)^{-2} Q(p)^* Q(p) = K(p)^* K(p) - \Id - R(p)^{-1} \diag (q(p))  - 2 R(p)^{-1} Q^{rev}(p) ,
    \end{align*}
    where $q$ is defined by~\cref{eq:q}.
    Using the definition of $q$, one can observe that $K(p)^* K(p) \leq \Id + R(p)^{-1} \diag (q(p)) $ and hence
    \begin{align*}
        Q(p)^* Q(p) \leq - 2 R(p) Q^{rev}(p) ,
    \end{align*}
    in the sense of positive semi-definite operators on $L^2(\pi)$.
    Choosing $f$ to be an eigenvector associated with $\gap( \E_\rho [R(p) Q^{rev}(p)] )$, we thus have
    \begin{align*}
        \mathrm{s}(\Hat{\sL}) \leq \sqrt{2 \gap( \E_\rho [ R(p) Q^{rev}(p) ] ) }.
    \end{align*}
    Finally, applying~\cite[][Lemma~10]{eberle2024non} and using that $\E_\rho [R(p) Q^{rev}(p)]$ is the generator of a $\pi$-reversible process gives
    \begin{align*}
        t_{rel}(\Hat{\sL})
        \geq \frac{1}{2 \mathrm{s}(\Hat{\sL})}
        \geq \frac{1}{2 \sqrt{2 \gap( \E_\rho [ R(p) Q^{rev}(p) ] )}}
        = \frac{1}{2 \sqrt{2}}  \sqrt{ t_{rel}(\E_\rho [ R(p) Q^{rev}(p) ]) } .
    \end{align*}
    The result in the particular case of $Q$ of the factorized form~\cref{eq:Q_factorized} follows by observing that, for $p \in \R^d$ and $y \neq x$, $Q^{rev}_{x,y}(p) \leq \| \nabla \log \rho(p) \| \Bar{Q}_{x,y}$ and hence
    \begin{align*}
        - R(p) Q^{rev}(p) \leq - \| \nabla \log \rho(p) \|^2 \Bar{R} \Bar{Q} .
    \end{align*}
\end{proof}
The above lower bound on the relaxation time is to be compared with similar results which can be obtained for \emph{non-reversible lifts} of reversible processes (cf.~\cite{eberle2024non}, in particular Theorem~11 and Remark~12).
The discrete Hamiltonian dynamics on the augmented space $\sX \times \R^d$ are indeed a \emph{first-order lift} of a stochastic process on the state space $\sX$, in the sense of~\cite{chen1999lifting}.
\begin{prop} \label{prop:lift}
    The Markov process generated by $\Hat{\sL}$ in~\cref{eq:discrete_HMC_generator} is a \emph{first-order lift} of the Markov process generated by $\E_\rho Q(p) \in \R^{\sX \times \sX}$.
    That is, for any $f,g \in L^2(\pi)$, it holds
    \begin{align*}
        \langle g \circ \tau, \, \Hat{\sL} (f \circ \tau) \rangle_{L^2(\Hat{\pi})} = \langle g, \, \E_\rho [ Q(p) ] f \rangle_{L^2(\pi)} ,
    \end{align*}
    where $\tau : \sX \times \R^d \to \sX$ is the projection on the first component.
\end{prop}

\begin{proof}
    First, observe that $\E_\rho Q(p) \in \R^{\sX \times \sX}$ is indeed a generator of a Markov process on $\sX$, since all properties of generators are preserved by taking expectation. 
    To check the first-order lift property, consider $f,g \in L^2(\pi) $.
    Then, from~\cref{eq:discrete_HMC_generator} and the integrability of $Q$, it follows
    \begin{align*}
        \langle g \circ \tau, \, \Hat{\sL} (f \circ \tau) \rangle_{L^2(\Hat{\pi})}
        & = \sum_{x \in \sX} \int_{\R^d} \left( \sum_{y \in \sX} Q_{x,y}(p) f(y) \right) g(x) \d \rho(p) \pi_x \\
        & = \sum_{x,y \in \sX}  \left(  \int_{\R^d}  Q_{x,y}(p)  \d \rho(p) \right) f(y) g(x) \pi_x \\
        & = \sum_{x,y \in \sX} \E_\rho [Q_{x,y}(p)] f(y) g(x) \pi_x .
    \end{align*}
\end{proof}

\section{Scaling limits: the lattice case}
\label{sec:scaling_limits_infill}

In this section, we study discrete Hamiltonian dynamics on a discretization of the Euclidean space $\R^d$, for dimension $d \geq 1$, with a square lattice.
We show in~\cref{thm:infill_scaling_limit} that discrete Hamiltonian dynamics on this lattice converge towards randomized Hamiltonian Monte Carlo processes as the discretization parameter tends to $0$.

\subsection{Scaling limit framework}
We will use the following~\cref{thm:scaling_limits_general}, which is a standard result on convergence of stochastic processes.
This framework was popularized in classical MCMC scaling limit results, such as the analysis of the random-walk Metropolis \cite{gelman1997weak} or the Metropolis-adjusted Langevin algorithm \cite{pillai2012optimal}, and will be used to prove our scaling-limit results in both the infill and hypercube cases.

\begin{theo}[{\cite[Chapter~4, Corollary~8.7]{ethier2009markov}}]
\label{thm:scaling_limits_general}
    Consider a Markov process $(Z^{(\infty)}_t)_{t \geq 0}$ on a space $\mathcal{Z}$ whose generator $\sL^{(\infty)}$ has a core $\sC$ containing an algebra which strongly separates points.
    Consider a sequence of Markov processes $(X^{(n)}_t)_{t\ge 0}$ on $\mathcal{X}^{(n)}$, for $n \geq 1$, with generator $\mathcal{L}^{(n)}$ and define the (not necessarily Markov) process $Z^{(n)}_t = \Pi_n(X^{(n)}_{g_n t})$ for some speed-up factor $g_n > 0$ and functions $\Pi_n \colon \mathcal{X}^{(n)} \to \mathcal{Z}$.
    Suppose that $Z^{(n)}_0 \to Z^{(\infty)}_0$ in distribution and that, for every $T \geq 0$ and every $\Bar{f} \in \sC$, there exists a sequence of subsets $(G_n)_{n \geq 1}$ of $\sX^{(n)}$ s.t.
    $$
    \lim_{n\to \infty} \PP
    \left[
    \forall t \in [0,T], \,
    X^{(n)}_{g_n t} \in G_n
    \right] = 1,
    $$
    and, denoting $f_n = \Bar{f} \circ \Pi_n$, we have $f_n \in \Dom(\sL^{(n)})$, $\sup_{n \geq 1} \| f_n \|_\infty < +\infty$ and
    $$
    \lim_{n \to \infty} \sup_{x \in G_n} \left| g_n  \mathcal{L}^{(n)} f_n (x) -  \mathcal{L}^{(\infty)} \Bar{f} (\Pi_n(x)) \right| = 0.
    $$
    Then, $(Z^{(n)}_t)_{t\ge 0} \xrightarrow{n \to +\infty} (Z^{(\infty)}_t)_{t\ge 0} $
    weakly in the Skorokhod topology. 
\end{theo}

\subsection{Infill lattice setting}

In the following, for $n \geq 1$ and some sequence of radii $(R_n)_{n \geq 1}$, we consider the state space $\sX^{(n)} \coloneq \sX_{n^{-1}, R_n}$ as defined in~\cref{eq:square_lattice}, with the associated directed graph structure and orientation of the edges.
Let $\pi \in \sP(\R^d)$ be a probability distribution on $\R^d$ with density $\pi \propto e^{-U}$ for some smooth function $U : \R^d \to \R$. We aim to sample from the target distribution $\pi^{(n)} \in \sP(\sX^{(n)})$ defined as the discretization of $\pi$, i.e.,
\begin{align*}
    \pi^{(n)}_x & \coloneq \frac{\pi(x) }{\sum_{y \in \sX^{(n)}} \pi(y) } , & x \in \sX^{(n)}.
\end{align*}
In the task of sampling from $\pi^{(n)}$, we consider the discrete Hamiltonian dynamics defined in~\cref{eq:discrete_HMC_generator} with a $d$-dimensional momentum.
Moreover, we consider a setting similar to that of \cref{subsec:factorized_generator}.
That is, given a $\pi^{(n)}$-reversible generator $\Bar{Q} \in \R^{\sX^{(n)} \times \sX^{(n)}}$, we consider
\begin{align*}
    & Q_{x,y}(p) =
    \Bar{Q}_{x,y} \max(0, - \langle \sigma_{x,y}, \nabla \log \rho(p) \rangle)
    & (x,y) \in E^{(n)}, p \in \R^d ,
\end{align*}
for some momentum distribution $\rho$ with smooth positive density.
We will assume that $\Bar{Q}$ is a locally-balanced generator, i.e., for $x \neq y \in \sX^{(n)}$,
\begin{align*}
    \Bar{Q}_{x,y} = G \left( \frac{\pi^{(n)}_y}{\pi^{(n)}_x} \right) \1((x,y) \in E^{(n)}),
\end{align*}
where $G$ is a balancing function satisfying $G(t) = t G(1/t)$. We also assume $G(1) = 1$, in order to standardize the global clock of the process (see e.g.\ \cite[][Remark~2.1]{livingstone2025foundations}).

In the above setting, the discrete Hamiltonian dynamics on $\sX^{(n)}$ have a generator $\Hat{\sL}^{(n)}$ defined in~\cref{eq:discrete_HMC_generator} and given for smooth functions $f : \sX^{(n)} \times \R^d \to \R$ by
\begin{equation}
\begin{split}
\Hat{\sL}^{(n)} f(x,p) = 
\sum_{\substack{1 \leq i \leq d \\ \eps \in \{-1,+1\} \\ x+\eps e_i/n \in \sX^{(n)}}} G\left(\frac{\pi(x+\eps e_i/n)}{\pi(x)}\right)\Big( \max\left(0,-\eps\,\partial_{p_i}\log\rho(p)\right) \left(f(x+\eps e_i/n,p)-f(x,p)\right) \\  + \eps\,\partial_{p_i}f(x,p)\Big) + n^{-1}\gamma\,\Hat{\sL}^D_p f(x,p). 
\end{split}
\label{eq:generator_Zd}
\end{equation}
where we consider an appropriate rescaling of the refreshment parameter as $n^{-1} \gamma$ for some $\gamma \geq 0$. 
As $n \to +\infty$, \cref{thm:infill_scaling_limit} below shows weak convergence of the discrete Hamiltonian dynamics towards standard randomized Hamiltonian Monte Carlo dynamics on $\R^d \times \R^d$.
It relies on the following set of assumptions.

\begin{asmp} \label{ass:infill}
    \hfill
    \begin{enumerate}
        \item The potential $U$ and the log density $\log \rho$ are of class $\sC^2$ with bounded second derivatives.
        In particular, this implies that for all $x, p \in \R^d$,
        \begin{align*}
            \left\| \nabla U(x) \right\| \leq C(1+\|x\|),
            \quad 
            \left\| \nabla \log \rho(p) \right\| \leq C(1+\|p\|),
        \end{align*}
        for some universal constant $C$.
        Also, $\rho$ has finite first moment and \hbox{$\lim_{t \to \infty} t \pi (\{x \colon \|x \| \geq t\}) = 0$}. 
        
        \item The balancing function $G$ is bounded and of class $\sC^2$ with bounded derivatives up to second order and with $G(1) = 1$.
        For example, this holds for $G : t \mapsto 2t / (1+t)$.
        
        \item The sequence of radii $(R_n)_{n \geq 1}$ is such that $R_n \to +\infty$ and $n^{-1} R_n \to 0$.
    \end{enumerate}
\end{asmp}

\begin{theo} \label{thm:infill_scaling_limit}
    Suppose~\cref{ass:infill} holds.
    For $n \geq 1$, let $(X^{(n)}_t, V^{(n)}_t)_{t\ge0}$ be the discrete Hamiltonian dynamics on $\sX^{(n)} \times \R^d$, with generator $\Hat{\sL}^{(n)}$ as in~\cref{eq:generator_Zd}, and assume the process is started in stationarity, that is $(X^{(n)}_0, V^{(n)}_0) \sim \pi^{(n)} \otimes \rho$.
    Let $(Z^{(\infty)}_{t}, P^{(\infty)}_t)_{t \geq 0}$ be a Markov process on $\R^d \times \R^d$ whose generator has $\sC^\infty_c(\R^d \times \R^d)$ as a core and is given by
    \begin{align*}
         \Hat{\sL}^{(\infty)}f(z,p)
         & = \langle -\nabla \log\rho(p), \nabla_z f(z,p)\rangle - \langle \nabla U(z), \nabla_p f(z,p)\rangle
         + \gamma \Hat{\sL}^D_p f(z,p),
         & f \in \sC^\infty_c(\R^d \times \R^d).
    \end{align*}
    Assume $(Z^{(\infty)}_0, P^{(\infty)}_0) \sim \pi \otimes \rho$.
    Then, considering the time-rescaled process  $(Z^{(n)}_{t}, P^{(n)}_t) := (X^{(n)}_{nt}, V^{(n)}_{n t})$, we have $(Z^{(n)}_{t}, P^{(n)}_t)_{t\ge 0} \xrightarrow{n \to +\infty} (Z^{(\infty)}_{t}, P^{(\infty)}_t)_{t\ge 0}$ weakly in the Skorokhod topology.
\end{theo}

\begin{proof}
    The proof will proceed by an application of~\cref{thm:scaling_limits_general}.
    One can check that, thanks to the appropriate rescaling of the refreshment parameter, the refreshment of the rescaled velocity $P^{(n)}_t$ exactly compensates for the refreshment of the limit process $P^{(\infty)}_t$ such that one can w.l.o.g. consider $\gamma = 0$.
    
    Using the regularity assumptions on $G$, for every index $i \in [d]$ and every $\eps \in \{-1, +1 \}$:
    \begin{align*}
        G \left( \frac{\pi(x+\eps e_i /n)}{\pi(x) } \right)
        &= G(1) - \eps n^{-1} G'(1) \partial_{x_i} U(x) + \sO(n^{-2}(1+\|x\|^2))\\
        &=1 - \eps (2n)^{-1} \partial_{x_i} U(x) + \sO(n^{-2}(1+\|x\|^2)) ,
    \end{align*}
    where we used $G(1)=1$ and $G'(1) = G(1)/2 = 1/2$. The latter follows by differentiating $G(t)=tG(1/t)$ at $t=1$.
    For a test function $f \in \sC^\infty_c(\R^d \times \R^d)$, we have:
    \begin{align*}
        f(x+\eps e_i / n, p) - f(x,p) = \eps n^{-1} \partial_{x_i} f(x,p) + \sO(n^{-2}).
    \end{align*}
    Let us denote by $\mathring{\sX}^{(n)} \coloneq \left\{ x \in \sX^{(n)} \colon |N_x| = 2d \right\}$ the interior of $\sX^{(n)}$.
    Then, for $(x,p) \in \mathring{\sX}^{(n)} \times \R^d$,
    \begin{align*}
        \Hat{\sL}^{(n)} f (x,p) =
        & \sum_{1 \leq i \leq d} G \left( \frac{\pi(x+ e_i /n)}{\pi(x)} \right)
        \max(0, - \partial_{p_i} \log \rho(p))
        (f(x+ e_i / n, p) - f(x,p)) \\
        & + \sum_{1 \leq i \leq d} G \left( \frac{\pi(x - e_i /n)}{\pi(x) } \right)
        \max(0, + \partial_{p_i} \log \rho(p))
        (f(x - e_i / n, p) - f(x,p)) \\
        & + \sum_{1 \leq i \leq d}  \partial_{p_i} f(x,p) \left[ G \left( \frac{\pi(x+ e_i /n)}{\pi(x) } \right) - G \left( \frac{\pi(x - e_i /n)}{\pi(x)} \right) \right] \\
        = & \sum_{1 \leq i \leq d} 
        \max(0, - \partial_{p_i} \log \rho(p))
        \left( 1
        + \sO(n^{-1}(1+\|x\|)) \right)
        \left( n^{-1} \partial_{x_i} f(x,p) + \sO(n^{-2}) \right) \\
        & + \sum_{1 \leq i \leq d} 
        \max(0, + \partial_{p_i} \log \rho(p))
        \left( 1
        + \sO(n^{-1}(1+\|x\|)) \right)
        \left( - n^{-1} \partial_{x_i} f(x,p) + \sO(n^{-2}) \right) \\
        & + \sum_{1 \leq i \leq d} \partial_{p_i} f(x,p) \left[ 
        - n^{-1}
        \partial_{x_i} U(x) + \sO(n^{-2}(1+\|x\|^2))  \right] \\
        = & n^{-1} \left[ - 
        \langle \nabla \log \rho(p), \nabla_x f(x,p) \rangle - 
        \langle \nabla U(x), \nabla_p f(x,p) \rangle \right]
        + \sO(n^{-2}) ,
    \end{align*}
    where we used that $f$ has compact support to remove the dependence w.r.t. $(x,p)$ in the remainder term.
Hence, we observe:
    \begin{align} \label{eq:generator_convergence_proof}
        \lim_{n \to +\infty}
        \sup_{(x,p) \in G_n}
        \left| n \Hat{\sL}^{(n)} f(x,p) - \Hat{\sL}^{(\infty)} f(x,p) \right| = 0,
    \end{align}
    where, for $n \geq 1$, $G_n \coloneq \mathring{\sX}^{(n)} \times \R^d$.
    Let $T \geq 0$.
    To apply \cref{thm:scaling_limits_general} (here we can simply consider the injection $\Pi_n : \sX^{(n)} \times \R^d \to \R^d \times \R^d$) and conclude the proof, it therefore suffices to show:
    \begin{align} \label{eq:probability_limit_proof}
        \lim_{n \to +\infty} \PP
        \left( \forall t \in [0, T], \, (X^{(n)}_{nt}, V^{(n)}_{nt}) \in G_n \right) = 1 .
    \end{align}
    In the following, for $n \geq 1$, we denote by $E_n$ the above event and
    introduce the events:
    \begin{align*}
        E_n^1 \coloneq
        \left\{
            \| X^{(n)}_{0} \|_\infty \leq R_n / 3
        \right\} ,
        \quad
        E_n^2 \coloneq
        \left\{
            N_n(nT)  \leq n  R_n / 3 
        \right\} ,
    \end{align*}
    where, for $t \geq 0$, $N_n(t)$ is the number of state space jumps occurring before time $t$.
    Note that, since $\| X^{(n)}_{nt} \|_\infty$ increases at most by $1/n$ when a state jump happens, it is clear that
    $E^1_n \cap E^2_n \subset E_n$ .
    We will show $\lim_{n \to +\infty} \PP \left( E^1_n \cap E^2_n \right) = 1$ using a union bound.
    First,
    we can use the stationarity of the process and Markov's inequality to obtain:
    \begin{align*}
        \PP [(E^2_n)^{c}]
        & \leq 3 \frac{\E [ N_n( n T ) ] }{n R_n} \\
        & = 3 T R_n^{-1} \E_{(x,p) \sim \pi^{(n)} \otimes \rho} [\sum_{y \in N_x} Q_{x,y}(p)] \\
        & \leq 6 d T \| G \|_\infty (\int_{\R^d} \|\nabla \log \rho\| \d \rho) R_n^{-1} ,
    \end{align*}
    where we used that each vertex $x \in \sX^{(n)}$ has at most $2d$ neighbours $y$ with $Q_{x,y}(p) \leq G(\frac{\pi(y)}{\pi(x)}) \|\nabla \log \rho(p)\|$.
    Hence $R_n \to +\infty$ implies $\PP [(E^2_n)^{c}] \to 0$.
    Then, using stationarity we have:
    \begin{align*}
        \PP [(E^1_n)^{c}]
        \leq \pi^{(n)}(\| x \|_\infty > R_n /3) .
    \end{align*}
    For $x, z \in \R^d$ with $\| x - z \|_\infty \leq \frac{1}{2n}$, we have $\pi(z) = \pi(x) (1 + \sO(n^{-1}(1+\|x\|)))$.
    Since by assumption $(R_n)_{n \geq 1}$ is a sequence of radii s.t. $R_n \to \infty$ and $n^{-1} R_n \to 0$, we have
    \begin{align*}
        \sum_{\substack{x \in n^{-1} \Z^d \\ \| x \|_\infty \leq R_n}} \pi(x)
        & = n^d (1 + \sO(n^{-1}R_n)) \sum_{\substack{x \in n^{-1} \Z^d \\ \| x \|_\infty \leq R_n}} \int_{\|x-z\|_\infty \leq 1/2n} \pi(z) \d z \\
        & = n^d (1 + \sO(n^{-1}R_n)) \int_{\|z\|_\infty \leq n^{-1} \lfloor n R_n \rfloor + \frac{1}{2n}} \pi(z) \d z \\
        & = n^d \pi( \| z \|_\infty \leq R_n) (1+o(1)) .
    \end{align*}
    Hence we obtain
    \begin{align*}
        \pi^{(n)}(\| x \|_\infty > R_n/3)
        = 1 - \frac{\sum_{\substack{x \in n^{-1} \Z^d \\ \| x \|_\infty \leq R_n/3}} \pi(x) }{\sum_{\substack{x \in n^{-1} \Z^d \\ \| x \|_\infty \leq R_n}} \pi(x) }
         &= 1 - \frac{\pi( \| z \|_\infty \leq R_n/3)}{\pi( \| z \|_\infty \leq R_n)} (1+o(1)) \\
        &= \sO( \pi( \| z \|_\infty > R_n/3 ) ).
    \end{align*}
    In particular, since $R_n \to +\infty$, this ensures $\PP [(E^1_n)^{c}] \to 0$.
\end{proof}

\section{Scaling limits: the hypercube case} \label{sec:scaling_limit_hypercube}

In this section, we study discrete Hamiltonian dynamics on the hypercube $\sX=\{0,1\}^n$ in the high-dimensional regime.
Our main objective is to detect a diffusive-to-ballistic speed-up relative to more standard and popular algorithms as the dimension $n$ tends to $+\infty$. 

For the heterogeneous target distributions given by  \cref{eq:target_with_hard_constraints} below, our theoretical and numerical results suggest the scaling laws summarized in~\cref{tab:table_mixing}. In particular, \cref{theo:scaling_limits_constrained_hypercube} shows that a low-dimensional projection of the discrete Hamiltonian dynamics converges to randomized HMC dynamics in continuous space, after an appropriate change of variable.
This scaling limit suggests that the discrete HMC algorithm can reduce the mixing time of reversible and standard non-reversible algorithms by a factor $n$.

\begin{table}[!ht]
\centering
\renewcommand{\arraystretch}{1.4} 
\setlength{\tabcolsep}{14pt}      

\begin{tabular}{|p{5cm}|p{3.5cm}|}
\hline
\textbf{Algorithm} & \textbf{Conjectured mixing time} \\ \hline
Discrete HMC $p \in \mathbb{R}$       & $\sO(n)$       \\ \hline
Discrete HMC $p \in \mathbb{R}^{n}$  & $\sO(n^{3/2})$ \\ \hline
Skew-reversible   MH           & $\sO(n^2)$     \\ \hline
Reversible MH                  & $\sO(n^2)$     \\ \hline
\end{tabular}
\caption{Mixing times of MCMC methods with the constrained target in \cref{eq:target_with_hard_constraints} on $\{0,1\}^{n}$ with $n=3m$. The mixing time for discrete HMC with $p \in \R$ is suggested by \cref{theo:scaling_limits_constrained_hypercube}. That of Reversible Metropolis-Hastings algorithms is implied by \cref{lm:basic_hypercube}. Those for discrete HMC with $p \in \R^{n}$ and Skew-reversible Metropolis-Hastings are suggested by numerical simulations (\cref{fig:ballistic_HMC_with_multidim_velocity}).}
\label{tab:table_mixing}
\end{table}

\subsection{Background}\label{sec:background_hypercube_mixing_times}

We start by reviewing standard results on the convergence of MCMC algorithms for sampling probability distributions on the hypercube.

\subsubsection{Limitations of reversible samplers}
The next lemma states some folklore facts about local kernels on the discrete hypercube, which can be useful to interpret results and set expectations.

In the following, we denote by $\delta_x$ the Dirac measure at a point $x \in \sX$ and by $\|\cdot\|_{\mathrm{TV}}$ the total variation distance on the set of finite measures.
Given a $\pi$-invariant Markov transition kernel $P$ on $\sX$, we define its total variation mixing time as $t_{mix}(\eps) \coloneq \inf\{t\geq 0\,:\, \|\delta_xP^t-\pi\|_{\mathrm{TV}}\leq \eps \hbox{ for all }x\in\sX\}$, for $\eps > 0$.

\begin{lm}\label{lm:basic_hypercube}
Let $\sX=\{0,1\}^n$,
$\pi\in\sP(\sX)$ and $P=(P_{x,y})_{x,y\in\sX}$ be a $\pi$-invariant transition kernel on $\sX$.
We say that $P$ is local if $P_{x,y}>0$ implies $|x-y| \leq 1$ (i.e., transitions modify at most a single bit).
We have:
\begin{enumerate}
\item[(a)] if $P$ is local then 
    $t_{mix}(\eps)\geq n/2$ for any $\eps < 1/2$.
\item[(b)] if $P$ is local and $\pi$-reversible, then 
$t_{mix}(\eps)\geq (2V(\pi)
-1)\log(1/(2\eps))$,
with $$V(\pi):=\sup_{v\in\R^n,\|v\|_\infty=1} \var_{X\sim \pi}(v^\top X)\,.$$
\end{enumerate}
\end{lm}

\begin{proof}
Part (a) follows from standard diameter bounds (see e.g.\ \cite[][Section~7.1.2]{levin2017markov}). Specifically, take $t<n/2$ and $x,y\in\sX$ with $x=(1,\ldots,1)-y$. Then locality of $P$ implies 
$\|\delta_xP^t-\delta_yP^t\|_{\mathrm{TV}}=1$ and thus $\max\{\|\delta_x P^t-\pi\|_{\mathrm{TV}},\|\delta_y P^t-\pi\|_{\mathrm{TV}}\}\geq 1/2$ as desired.

Consider now part (b).
As above, denote the spectral gap of $P$ as
$\gap(P)=\inf\langle (I-P)f,f\rangle_{L^2(\pi)}/\var_\pi(f)$ with infimum running over non-constant functions $f:\sX\to \R$.
Classical results \cite[Theorem~12.5]{levin2017markov} imply $t_{mix}(\eps)\geq (1/\gap(P)-1)\log(1/(2\eps))$ for all $\eps <1/2$.
For a linear test function, $f_v(x)=\sum_{i=1}^nv_ix_i$ for $v\in\R^n$, using reversibility and locality of $P$, we obtain $\langle (I-P)f_v,f_v\rangle_{L^2(\pi)}\leq \frac{1}{2}\max_i v_i^2= \frac{1}{2}\|v\|_\infty^2$.
Taking the infimum over $v\in\R^n$ with $\|v\|_\infty=1$, we obtain 
$$
\gap(P)\leq \inf_{\|v\|_\infty=1}\frac{\|v\|_\infty^2}{2 \var_{X\sim \pi}(v^\top X)}
=\left(2
V(\pi)
\right)^{-1}
$$
and $t_{mix}(\eps)\geq (2V(\pi)
-1)\log(1/(2\eps))$ as desired.
\end{proof}

\Cref{lm:basic_hypercube} implies that any local kernel takes at least $\sO(n)$ steps to mix. If in addition $P$ is reversible, it takes at least $\sO(V(\pi))$ steps to mix. For distributions that are flat in some direction $V(\pi)$ can be as large as $\sO(n^2)$.

\subsubsection{Example: symmetric target distribution}

To illustrate the above, consider the following toy target distribution
\begin{align} \label{eq:unconstrained_target_hypercube}
\pi_x &\propto |x|!(n-|x|)! \, , 
&x\in\{0,1\}^n\,.
\end{align}
The above is a canonical prior distribution in Bayesian models with binary latent variables, such as variable selection or mixture models, and is used to enforce a uniform prior on the number of active variables $|x|=\sum_ix_i$.
The covariance matrix of $\pi$ in \cref{eq:unconstrained_target_hypercube} is ill-conditioned, with one large eigenvalue in the $(1,\dots,1)\in\R^n$ direction and one small eigenvalue of multiplicity $n-1$ on the complement of $(1,\dots,1)$. In particular, one has $V(\pi)=\sO(n^2)$ since $\var_{X\sim \pi}(v^\top X)=\sO(n^2)$ for $v=(1,\dots,1)$.
Thus, by \cref{lm:basic_hypercube}, any local reversible kernel takes at least $\sO(n^2)$ iterations to mix, while local non-reversible ones can in principle accelerate by a $\sO(n)$ factor.

The distribution in \eqref{eq:unconstrained_target_hypercube} lends itself naturally to non-reversible samplers proposed in the literature. 
For example, applying standard skew-reversible MH schemes (see~\cite[][Algorithm~1 and Section~4.1]{gagnon2024asymptotic}) with the set of directions defined in \cref{eq:1d_momentum_hypercube} leads to a Markov kernel that achieves the optimal $\mathcal{O}(n)$ speed-up.
This is illustrated in the following proposition, which follows directly from the classical results in \cite{diaconis2000analysis} after applying the projection argument discussed in \cref{lemma:reduction} in the Appendix.
We recall that we consider the graph structure described in~\cref{sec:example_discrete_hypercube}.
In particular, for $x \in \{0,1\}^n$, $N_x = \{ y \, \colon \, |y-x|=1\}$, $N_x^+ = \{ y \in N_x \, \colon \, |y| = |x|+1 \}$ and $N_x^- = \{ y \in N_x \, \colon \, |y| = |x|-1 \}$.

\begin{prop}\label{prop:non_rev_unif}
Consider the non-reversible Markov chain $(X_i, V_i)_{i=0,1,2,\ldots}$ with state space $\{0,1\}^n\times\{-,+\}$, invariant distribution $\Hat \pi(x,v) \propto |x|!(n-|x|)!$,
and transition kernel $K$ defined, for $x,y \in \{0, 1\}^n$, as
\begin{equation}
\begin{aligned} \label{eq:non_rev_uniform}
K((x,+), (y,+)) 
&=
\frac{n}{n+1}
    \frac{\1(y \in N_x^+)}{n-|x|} + \frac{\1(|x|=|y|=n)}{n+1}
    \,,\\ 
K((x,-), (y,-)) 
&=
\frac{n}{n+1}
    \frac{\1(y \in N_x^-)}{|x|} + \frac{\1(|x|=|y|=0)}{n+1}
    \, , \\ 
K((x,+),(y,-))
& = \frac{1}{n+1} \frac{\1(y \in N_x^+)}{n-|x|} \, , \\
K((x,-),(y,+))
& = \frac{1}{n+1} \frac{\1(y \in N_x^-)}{|x|} \, , \\
K((x,v), (x,-v)) 
&=
\frac{n}{n+1} \1(v=-,|x|=0)
+
\frac{n}{n+1} \1(v=+,|x|=n)
\, .
\end{aligned}
\end{equation}

    Let $\mu\in\sP(\{0,1\}^n\times\{-,+\})$ be invariant under permutations of the coordinates of $x$, e.g., $\mu=\delta_{(0,\ldots,0,+)}$ or $\mu=\delta_{(0,\dots,0,-)}$. 
    Then  
    $\|\mu K^t-\hat{\pi}\|_{\mathrm{TV}}\leq (1-C)^{\lfloor t/(4(n+1))\rfloor}$ with $C$ being some positive universal constant independent of $n$.
\end{prop}
\begin{proof}
The result follows by combining \cref{lemma:reduction} with $\phi(x,v)=(|x|,v)$ and \cite[][Theorem~1]{diaconis2000analysis}.
Specifically, let $\sX=\{0,1\}^n\times\{-,+\}$, $E=\{0,\dots,n\}\times\{-,+\}$,  $\phi:\sX \to E$ be defined as $\phi(x,v)=(|x|,v)$, and let $\mu\in\sP(\sX)$ be invariant under permutations. 
Then, the kernel in \cref{eq:non_rev_uniform} satisfies the conditions in \cref{eq:Markov_proj} and \cref{eq:intertwining} of \cref{lemma:reduction}, with $\hat{\pi}$ instead of $\pi$ in \eqref{eq:intertwining}, and $\mu$ satisfies the condition in point (b) of \cref{lemma:reduction}.
Thus, by \cref{lemma:reduction}, we have
$\|\mu K^t-\hat{\pi}\|_{\mathrm{TV}}
=\|\bar\mu \bar{K}^t-\bar{\pi}\|_{\mathrm{TV}}$, with $\bar{\pi}=\phi_{\#}\hat{\pi}$ being the uniform distribution on $E$, and 
$\bar{K}$ being the kernel defined in \cite[eq.~(2.2)]{diaconis2000analysis}, with $n+1$ instead of $n$.
The bound $\|\bar\mu \bar{K}^t-\bar{\pi}\|_{\mathrm{TV}}\leq (1-C)^{\lfloor t/(4(n+1))\rfloor}$ then follows directly from \cite[][Theorem~1]{diaconis2000analysis}.
\end{proof}

\Cref{prop:non_rev_unif} implies that
the mixing times of $K$ in \cref{eq:non_rev_uniform} scale as $\mathcal{O}(n)$ as $n\to\infty$.

\subsubsection{Example: non-symmetric target distribution}

The symmetric target in \cref{eq:unconstrained_target_hypercube} is, however, not a good representative model for many realistic applications (e.g.\ Bayesian posteriors with data).
This is because $\pi$ in \cref{eq:unconstrained_target_hypercube} is only a function of $|x|$ and is exchangeable w.r.t.\ permutations of coordinates, while commonly encountered distributions do not have such a high degree of symmetry.
In order to introduce asymmetry, we thus consider here another toy, yet more representative, distribution $\pi^{(3m)}$ on the $3m$-dimensional hypercube $\sX = \{0,1\}^{3m}$, for an integer $m \geq 1$, s.t.
\begin{equation}
\label{eq:target_with_hard_constraints}
    \pi^{(3m)}_x  \propto \binom{m}{|x|_2}^{-1} \prod_{1 \leq i \leq m} \1(x_i=0) \prod_{2m < i \leq 3m} \1(x_i=1),
\end{equation}
where we denote $|x|_2 \coloneq \sum_{m < i \leq 2m} |x_i|$.
Here, heterogeneity is encoded by hard constraints, thus restricting the support of the distribution and breaking the symmetry in \cref{eq:unconstrained_target_hypercube}.
In particular, $\pi^{(3m)}$ is now supported on those $x \in \sX$ s.t. $x_i = 0$ for $i \leq m$, $x_i = 1$ for $i > 2m$ and, under $\pi^{(3m)}$, the marginal distribution of $|x|_2$ is uniform on $\{0, \ldots, m\}$.
This is more realistic because a Metropolis-Hastings algorithm with an uninformed proposal would not have acceptance probability equal to $1$. In particular, the analogue of the non-reversible Markov kernel considered in \cref{eq:non_rev_uniform} has $\sO(1)$ probability to switch the velocity component at each iteration.
We expect this to break its ballistic behaviour ($\sO(m)$ mixing time), leading to a diffusive behaviour ($\sO(m^2)$ mixing time), in analogy with the stylized models in \cite[][Section 2]{roberts2025quantifying}. See \cref{fig:ballistic_HMC_with_multidim_velocity} below for a numerical illustration of this phenomenon. 
In contrast, our scaling limit result in~\cref{theo:scaling_limits_constrained_hypercube} supports the fact that the discrete Hamiltonian dynamics introduced in~\cref{sec:discrete_hamiltonian_dynamics} are able to preserve the ballistic behaviour and achieve $\sO(m)$ mixing time.

\subsection{Scaling limits on the hypercube with $1$-dimensional momentum}
\label{sec:scaling_limit_hypercube_1d}

We study scaling limits of the discrete Hamiltonian dynamics on the hypercube with a $1$-dimensional velocity targeting the constrained target $\pi^{(3m)}$ in \eqref{eq:target_with_hard_constraints}.
In this setting, exploiting the prior knowledge of an ill-conditioned direction in the target leads to a diffusive-to-ballistic speed-up w.r.t. reversible samplers. 
\Cref{theo:scaling_limits_constrained_hypercube} shows that, in the high-dimensional limit, the discrete Hamiltonian dynamics converge towards a randomized HMC process on a one-dimensional space with a time rescaling of order $\sO(m)$, $m$ being the dimension of the unconstrained component of the target.
The intuition is that, thanks to the structure of \eqref{eq:target_with_hard_constraints}, the dynamics can be understood by looking only at the induced one-dimensional dynamics on $|x|_2$ (i.e., the number of bits on the set of unconstrained indices), which is what we perform the scaling limit on.

We consider the discrete Hamiltonian dynamics defined in~\cref{eq:discrete_HMC_generator} on the hypercube $\sX = \{0,1\}^{3m}$, for $m \geq 1$, with the constrained target $\pi^{(3m)}$ defined in~\cref{eq:target_with_hard_constraints} and with a $1$-dimensional momentum and directions $\sigma_{x,y}$ as in~\cref{eq:1d_momentum_hypercube}.
Similarly to \cref{subsec:factorized_generator}, we consider a generator $Q$ of the form
\begin{align*}
    & Q_{x,y}(p) =
    \Bar{Q}_{x,y} \max(0, -  \sigma_{x,y} \nabla \log \rho(p)) ,
\end{align*}
where $\Bar{Q}$ is a locally-balanced generator of the form of either~\cref{eq:Q_bar_LB} or~\cref{eq:Q_bar_LB_weighted}, for some balancing function $G$.

In the above setting, the discrete Hamiltonian dynamics on $\sX = \{0,1\}^{3m}$ have a generator $\Hat{\sL}^{(m)}$ given for smooth functions $f : \sX \times \R \to \R$ by
\begin{equation}
\begin{aligned} \label{eq:generator_hypercube}
    \Hat{\sL}^{(m)} f(x,p)
    = & \sum_{y \in N_x^+} \Bar{Q}_{x,y} 
    \left[
    \max(0, - \nabla \log \rho (p)) ( f(y, p) - f(x,p) )  + \nabla_p f(x,p)
    \right] \\
    & + \sum_{y \in N_x^-} \Bar{Q}_{x,y} 
    \left[
    \max(0, \nabla \log \rho (p)) ( f(y, p) - f(x,p) )  - \nabla_p f(x,p)
    \right] \\
    & + m^{-1} \gamma \Hat{\sL}^D_p f(x,p) ,
\end{aligned}
\end{equation}
where we consider an appropriately rescaled refreshment parameter $m^{-1} \gamma$ for some $\gamma \geq 0$.

We rely on the following assumptions.

\begin{asmp} \label{ass:hypercube}
    \hfill
    \begin{enumerate}
        \item The log density $\log \rho$ is of class $\sC^2$ with bounded second derivatives and $\rho$ has finite first moment. 
        
        \item The balancing function $G$ is bounded and of class $\sC^2$ with bounded derivatives up to second order.
        For example, this holds for $G : t \mapsto 2t / (1+t)$.
    \end{enumerate}
\end{asmp}

\begin{theo}
\label{theo:scaling_limits_constrained_hypercube}
     Suppose~\cref{ass:hypercube} holds.
     For $m \geq 1$, let $(X^{(m)}_t, V^{(m)}_t)_{t\ge0}$ be the discrete Hamiltonian dynamics on $\{0,1\}^{3m}\times\R$, with the directions $\sigma_{x,y}$ as in \cref{eq:1d_momentum_hypercube}, the target $\pi^{(3m)}$ as in~\cref{eq:target_with_hard_constraints} and the generator $\Hat{\sL}^{(m)}$ given by~\cref{eq:generator_hypercube}.
     Assume the process is started in stationarity, i.e., $(X^{(m)}_{0}, V^{(m)}_0) \sim \pi^{(3m)} \otimes \rho$.
    Let $(Z^{(\infty)}_t, P^{(\infty)}_t)_{t \geq 0}$ be a Markov process on $[0,1] \times \R$ whose generator has $\sC^\infty_c([0,1] \times \R)$ as a core and is given by:
    \begin{align}
    \label{eq:scaling_limits_constr.hypercube}
        \Hat{\sL}^{(\infty)} f(z,p) & = - h(z) \nabla \log \rho(p) \nabla_z f(z,p) + h'(z) \nabla_p f(z,p) + \gamma \Hat{\sL}^D_p f(z,p),
        & f \in \sC^\infty_c([0,1] \times \R),
    \end{align} 
    where, for $z \in [0,1]$,
    \begin{align*}
        h(z) \coloneq
        \begin{cases}
            \frac{1-z}{3} G( \frac{z}{1-z}) & \hbox{if $\Bar{Q}$ is as in~\cref{eq:Q_bar_LB}} \\
            \frac{1-z}{2-z} G( \frac{z}{1+z} \frac{2-z}{1-z}) & \hbox{if $\Bar{Q}$ is as in~\cref{eq:Q_bar_LB_weighted}}
        \end{cases}
        \, .
    \end{align*}
    Assume $(Z^{(\infty)}_0, P^{(\infty)}_0) \sim \mathcal{U}([0,1]) \otimes \rho$.
    Then, considering the rescaled process  $(Z^{(m)}_{t}, P^{(m)}_t)_{t \geq 0} \coloneq \left( m^{-1} \left|X^{(m)}_{m t} \right|_2, V^{(m)}_{m t} \right)_{t \geq 0}$,
    we have
    $
    (Z^{(m)}_t, P^{(m)}_t)_{t \ge 0} \xrightarrow{m \to +\infty} (Z^{(\infty)}_{t}, P^{(\infty)}_t)_{t\ge 0} 
    $
    weakly in the Skorokhod topology.
\end{theo}

\begin{proof}
    The proof will proceed by an application of~\cref{thm:scaling_limits_general}.
    One can check that, thanks to the appropriate rescaling of the refreshment parameter, the refreshment of the rescaled velocity $P^{(m)}_t$ exactly compensates for the refreshment of the limit process $P^{(\infty)}_t$ such that one can w.l.o.g. consider $\gamma = 0$.
    For simplicity, we will also focus on the case with $\bar Q_{x,y}$ as in \cref{eq:Q_bar_LB_weighted}, the case with $\bar Q_{x,y}$ as in \cref{eq:Q_bar_LB} being analogous.
    
    In the following, let us fix a test function $\Bar{f} \in \sC^{\infty}_c([0,1] \times \R)$ and, for $m \geq 1$, denote $f_m(x,v) \coloneq \Bar{f}(z, v)$ with $z = |x|_2/m \in [0,1]$.
    Then, for $\Hat{\sL}^{(m)}$ the generator of the process $(X^{(m)}_t, V^{(m)}_t)_{t \geq 0}$, we have $f_m \in \Dom(\Hat{\sL}^{(m)})$ and, by~\cref{eq:generator_hypercube},
    \begin{align*}
        \Hat{\sL}^{(m)} f_m(x, v) 
        = & (-\nabla\log\rho(v))_+ h^{+}_m(|x|_2/m)( \Bar{f}(|x|_2/m + m^{-1},v) - \Bar{f}(|x|_2/m,v)) \\
        & + (\nabla\log\rho(v))_+ h^{-}_m(|x|_2/m) ( \Bar{f}(|x|_2/m - m^{-1},v) - \Bar{f}(|x|_2/m,v) ) \\
        & + (h_m^{+}(|x|_2/m) - h_m^{-}(|x|_2/m))\nabla_v \Bar{f}(|x|_2/m,v) ,
    \end{align*}
    with, for $z \in [0,1]$ and any $x \in \supp(\pi^{(3m)})$ s.t. $z=|x|_2/m$,
   \begin{align}
        \label{eq:hm+}
    h_m^+(z) & \coloneq \sum_{y \in N^+_x} \Bar{Q}_{x,y}
    = \frac{1-z}{2-z} G( \frac{z+m^{-1}}{1-z} \frac{2-z}{1+z+m^{-1}}), \\
         \label{eq:hm-}
    h_m^-(z) & \coloneq \sum_{y \in N^-_x} \Bar{Q}_{x,y} = \frac{z}{1+z} G( \frac{1-z+m^{-1}}{z} \frac{1+z}{2-z+m^{-1}}) .
\end{align}
Since $G$ has bounded derivatives up to second order, introducing the functions \hbox{$\alpha(z) = z/(1+z)$} and $\beta(z) = (1-z)/(2-z)$, one can check that, for $z \in (0,1)$,
\begin{align*}
    h_m^+(z) & = \beta(z) G( \frac{\alpha(z)}{\beta(z)} ) + m^{-1} G'(\frac{\alpha(z)}{\beta(z)}) \alpha'(z) + \sO(m^{-2} (1-z)^{-1}) , \\
    h_m^- (z) & = \alpha(z) G( \frac{\beta(z)}{\alpha(z)} ) - m^{-1} G'(\frac{\beta(z)}{\alpha(z)}) \beta'(z) + \sO(m^{-2} z^{-1}) .
\end{align*}
In particular, using the balancing property of $G$,
\begin{equation}
    \label{eq:hm_asymptotics}
    \begin{aligned}
        h_m^+(z)
        & = \frac{1-z}{2-z} G( \frac{z}{1+z} \frac{2-z}{1-z}) + \sO(m^{-1} + m^{-2} (1-z)^{-1}) \\
        & = h(z) + \sO(m^{-1} + m^{-2} (1-z)^{-1}), \\
        h_m^-(z)
        & = \frac{z}{1+z} G( \frac{1+z}{z} \frac{1-z}{2-z}) + \sO(m^{-1} + m^{-2} z^{-1}) \\
        & = h(z) + \sO(m^{-1} + m^{-2} z^{-1}), \\
        m ( h_m^+(z) - h_m^-(z))
        & = h'(z) + \sO(m^{-1}z^{-1}(1-z)^{-1}) .
    \end{aligned}
\end{equation}
Since $\Bar{f} \in \sC^\infty_c([0,1] \times \R)$, we have $\Bar{f}(z + m^{-1}, p) = \Bar{f}(z, p) + m^{-1} \nabla_z \Bar{f}(z,p) + \sO(m^{-2})$ and using the previous calculations we get, for every $(x,v) \in \supp(\pi^{(3m)}) \times \R$, denoting $z = |x|_2/m \in (0,1)$,
\begin{align*}
    m \Hat{\sL}^{(m)} f_m(x,v)
    & =
    - h(z) \nabla \log \rho(v) \nabla_z \Bar{f}(z,v) + h'(z) \nabla_v \Bar{f}(z,v)
     + \sO(m^{-1} z^{-1} (1-z)^{-1}).
\end{align*}
Let $T \geq 0$ and, for $m \geq 1$, define
\begin{align*}
    G_m \coloneq
    \left\{
        (x,v) \in \{0,1\}^{3m} \times \R \, \colon \, x \in \supp(\pi^{(3m)}), \, |x|_2 \in (\eps_m, m-\eps_m)
    \right\} ,
\end{align*}
with $\eps_m = \lfloor m^{2/3} \rfloor$.
It follows from the previous equality that
\begin{align*}
    \lim_{m \to +\infty} \sup_{(x,v) \in G_m}
    \left|
    m \Hat{\sL}^{(m)} f_m(x,v) - \Hat{\sL}^{(\infty)} \Bar{f} (|x|_2/m, v)
    \right|
    = 0.
\end{align*}
It then suffices to show
\begin{align*}
    \lim_{m \to +\infty} \PP
    \left[
    \forall t \in [0,T], (X^{(m)}_{m t}, V^{(m)}_{m t} ) \in G_m
    \right]
    = 1 
\end{align*} 
to apply~\cref{thm:scaling_limits_general} and conclude the proof.
For $m \geq 1$, let us denote by $E_m$ the above event and introduce the events
\begin{align*}
    E^1_m  & \coloneq
    \left\{
    X^{(m)}_0 \in \supp(\pi^{(3m)}), \,
    \left| X^{(m)}_0 \right|_2 \in ( 2 \eps_m, m- 2 \eps_m)
    \right\} , \\
    E^2_m  & \coloneq
    \left\{
    N_m^-(mT) <  \eps_m
    \right\} ,
    \quad 
    E^3_m  \coloneq
    \left\{
    N_m^+(mT) <  \eps_m
    \right\} ,
\end{align*}
where, for $t \geq 0$, $N_m^-(t)$ (resp. $N_m^+(t)$) is the number of jumps before time $t$ that remove (resp. add) one bit while $|X^{(m)}|_2 \in (\eps_m, 2 \eps_m]$ (resp. while $|X^{(m)}|_2 \in [m-2\eps_m, m- \eps_m)$).
Clearly, by construction, we have $E^1_m \cap E^2_m \cap E^3_m \subset E_m$.
Using a union bound, we will show $\PP [ E^1_m \cap E^2_m \cap E^3_m] \to 1$.
First, since $X^{(m)}_0 \sim \pi^{(3m)}$ we have $\left| X^{(m)}_0 \right|_2 \sim \mathcal{U}(\{0, \ldots, m\})$ and
\begin{align*}
    \PP \left[ (E^1_m)^c \right] = \sO(m^{-1} \eps_m) \xrightarrow{m \to +\infty} 0. 
\end{align*}
Then, observe that $(N^-_m(t))_{t \geq 0}$ is a counting process with intensity
$$
\lambda(t) = \1_{|X^{(m)}_t|_2 \in(\eps_m, 2 \eps_m] } \sum_{y \in N_{X^{(m)}_t}^-} Q_{X^{(m)}_t, y}(V^{(m)}_t).
$$
Hence, by stationarity we get
\begin{align*}
    \E [ N_m^-(mT) ]
    & = m T \, \E_{(x, v) \sim \pi^{(3m)} \otimes \rho}
    \left[
    \1_{|x|_2 \in(\eps_m, 2 \eps_m] } \sum_{y \in N_{x}^-} Q_{x, y}(v)
    \right] \\
    & \leq m T ( \int_\R \| \nabla \log \rho \| \d \rho)
    \, \E_{x \sim \pi^{(3m)}}
    \left[
    \1_{|x|_2 \in(\eps_m, 2 \eps_m] } \sum_{y \in N_{x}^-} \Bar{Q}_{x, y}
    \right] \\
    & = m T ( \int_\R \| \nabla \log \rho \| \d \rho)
    \, \E_{x \sim \pi^{(3m)} }
    \left[
    \1_{|x|_2 \in(\eps_m, 2 \eps_m] } h_m^-(|x|_2/m)
    \right],
\end{align*}
where we used that $Q_{x,y}(v) \leq \Bar{Q}_{x, y} \| \nabla \log \rho(v) \|$ and the definition of $h^-_m$ in~\cref{eq:hm-}.
Now, using that $h^-_m(z) = h(z) + \sO(m^{-1})$ with $h(z) = \sO(z)$ as $z \to 0$ (this comes from the fact that $G$ is bounded and balanced), we find $\E [ N_m^-(mT) ] = \sO(m^{-1} \eps_m^2 )$.
Hence, by Markov's inequality,
\begin{align*}
    \PP [ (E^2_m)^c ] = \PP [ N_m^-(mT) \geq \eps_m ] = \sO( m^{-1} \eps_m)
    \xrightarrow{m \to +\infty} 0.
\end{align*}
By a similar argument, one can show $\PP [ (E^3_m)^c ] \xrightarrow{m \to +\infty} 0$, concluding the proof.
\end{proof}

\begin{rmk}
\label{remark:hamiltonian_dynamics_scaling_hypercube}
Between velocity refreshments, the limiting process $(Z^{(\infty)}_t, P^{(\infty)}_t )_{t \ge 0}$ actually corresponds to a reparametrization of the classical Hamiltonian dynamics on $\R$.
Indeed, letting $Q_t = \Phi(Z^{(\infty)}_t)$, with $\Phi(z) = \int^z_{1/2} h^{-1}(u) \d u$, then 
$$
\d Q_t = \Phi'(Z^{(\infty)}_t) \d Z^{(\infty)}_t =  -\nabla \log \rho(P^{(\infty)}_t) \d t, \qquad \d P^{(\infty)}_t = h'(\Phi^{-1}(Q_t)) \d t.
$$
Thus $(Q_t, P^{(\infty)}_t)$ coincides with the standard Hamiltonian dynamics with Hamiltonian $U(q) + K(p)$, kinetic energy $K(p) = - \log \rho(p)$ and potential energy $U(q) = -\log h(\Phi^{-1}(q))$.
\end{rmk}

\subsection{Scaling limit with multi-dimensional momentum}
\label{sec:scaling_limit_hypercube_nd}

In practice, the choice of direction in \cref{eq:1d_momentum_hypercube} requires a priori knowledge of the direction(s) where the target distribution has large variability, such as $v=(1,\ldots,1)$ in the case of~\cref{eq:unconstrained_target_hypercube,eq:target_with_hard_constraints}.
Depending on the context, such information might however not be available in practice. 
We thus consider the case of the discrete Hamiltonian dynamics on the hypercube $\sX = \{0,1\}^n$ with $n$-dimensional momentum $p\in\R^n$ and $\sigma_{x,y}$ as in \cref{eq:n_momentum_hypercube}, which does not require any a priori knowledge of special directions and is generally applicable.
For the constrained hypercube in \cref{eq:target_with_hard_constraints}, numerical experiments show that an averaging phenomenon occurs in the velocity components, allowing the discrete Hamiltonian dynamics to still exhibit a ballistic behaviour but with a $\sO(\sqrt{m})$ slowdown.
This leads to a $\sO(m^{3/2})$ mixing time which still improves on the mixing time of reversible samplers. 
\Cref{fig:ballistic_HMC_with_multidim_velocity} shows the estimates of the Effective Sample Size (ESS) relative to the test function  $f(x) = |x|$,
as the dimensionality of the hypercube grows, averaged across 10 independent simulations for four different MCMC algorithms: the discrete HMC algorithm (\cref{alg:HMC_continuous_momentum}) with both $p \in \R$ and $p \in \R^{3m}$, the skew-reversible Metropolis-Hastings algorithm \cite[][Algorithm~1]{gagnon2024asymptotic} and its reversible counterpart. 
See \cref{sec:future_res} for more discussion.

\begin{figure}[!ht]
    \centering  \includegraphics[width=0.32\linewidth]{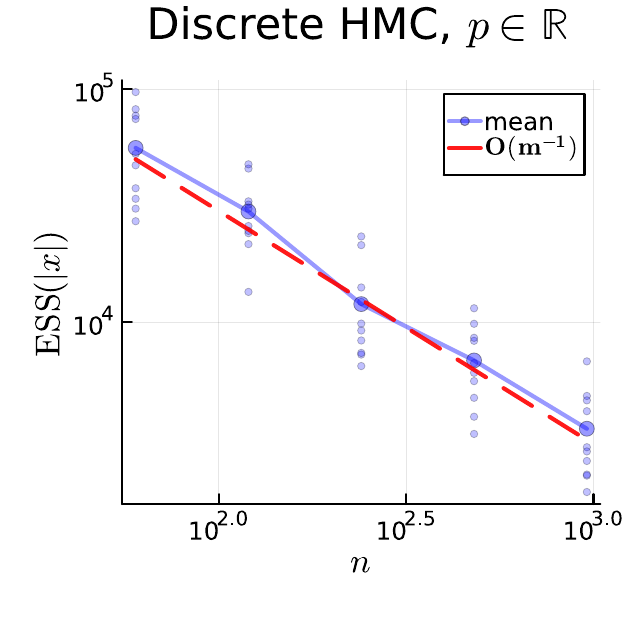}
    \includegraphics[width=0.32\linewidth]{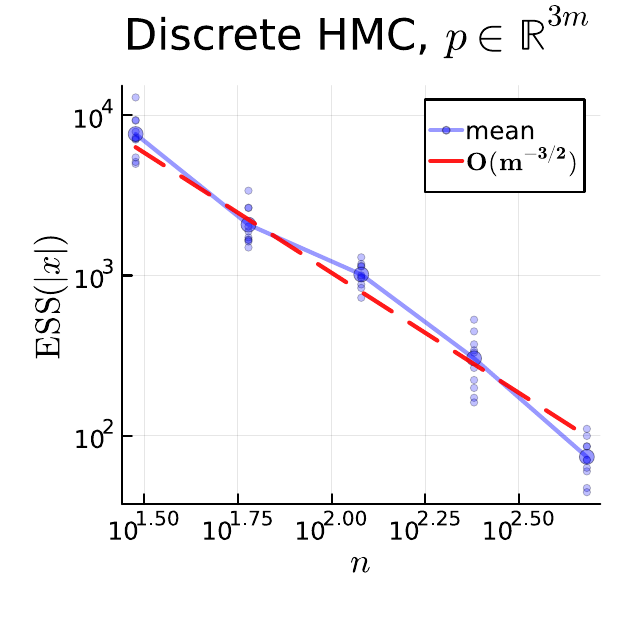}
    \includegraphics[width=0.32\linewidth]{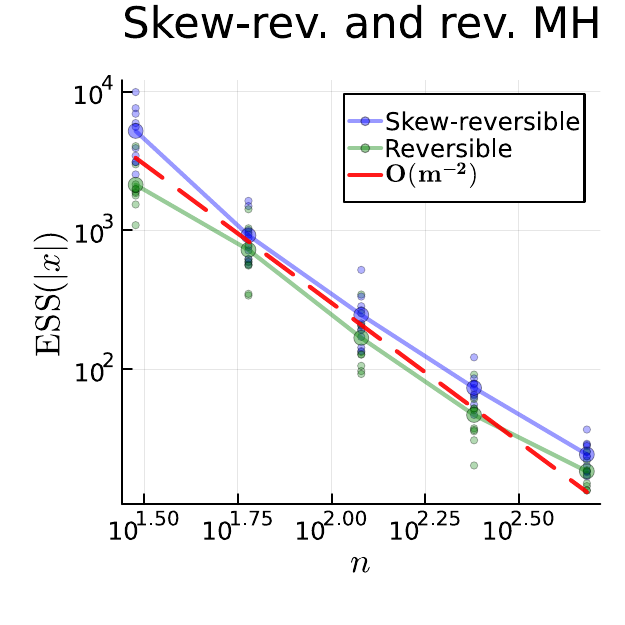}
    \caption{Average ESS ($y$-axis, on a log-scale) estimated with 10 independent simulations of discrete HMC with $p \in \R $ (left panel), $p \in \R^{3m}$ (middle panel), and skew-reversible and reversible MH algorithms (right panel) for different values of $n = 3m$ ($x$-axis, on a log-scale) and refreshment rate $\gamma = 1/m$. Red dashed lines indicate the expected slopes.}
    \label{fig:ballistic_HMC_with_multidim_velocity}
\end{figure}

\section{Exact simulation schemes and faster discretizations} \label{sec:implementation}

\subsection{Exact simulation scheme: implementation details} \label{sec:implementation_details}

The dynamics described in \cref{example:gauss_and_laplace_momentum_hmc} are computationally appealing since the momentum variable evolves linearly between two state-space jumps.
This allows for a simple and exact simulation of the associated stochastic process in continuous time (without discretization error), which is akin to piecewise deterministic Monte Carlo; see \cref{alg:HMC_continuous_momentum}. Here, for a rate function $\lambda \colon \mathbb{R}_+ \to \mathbb{R}_+$, $\tau \sim \text{IPP}(\lambda)$ 
denotes the first event time of an inhomogeneous Poisson process with rate $\lambda$, i.e., 
$$
\mathbb{P}(\tau > t) = \exp\left(- \int_0^t\lambda(s) \d s\right), \quad t \ge 0.
$$
  
Note that in \cref{alg:HMC_continuous_momentum}, an event time of some inhomogeneous Poisson process is simulated for each $y$ in the neighbourhood of the current state $X_t$. The next state is then selected by taking the $\textrm{argmin}$ of those random times, a procedure whose correctness relies on standard results for independent Poisson processes; see e.g.\ \cite[][Section 2.6]{norris1998markov}.

\begin{algorithm}
\caption{\label{alg:HMC_continuous_momentum} Discrete HMC}
    \begin{algorithmic}[1]
        \Statex Requires: a way to sample $\tau \sim \text{IPP}(t \mapsto \max(0, - \langle \sigma, \nabla \log \rho(p + t a) \rangle))$, for any $\sigma, p, a \in \R^d$, and a refreshment rate $\gamma\ge0$.
        \Statex Initialize $t=0$ and $(X_t,P_t) \in \sX \times \R^d$.
        \While{Stopping criterion is not met}
            \State Set $A = \sum_{y \in N_{X_t}} \Bar{Q}_{X_t,y} \sigma_{X_t,y}$.
            \For{$y \in N_{X_t}$} 
            \Comment{This loop can be executed in parallel}
                \State Sample $\tau_y \sim \text{IPP}(s \mapsto \max(0, - \Bar{Q}_{X_t,y} \langle \sigma_{X_t,y}, \nabla \log \rho (P_t + s A) \rangle ) )$.
            \EndFor
            \State Set $\tau_{\text{ref}}=1/\gamma$ (or alternatively sample $\tau_{\text{ref}} \sim \text{Exp}(\gamma)$). 
            \State Select $y_* \in \textrm{argmin}_{y \in N_{X_t}} \tau_y$ and set $\tau = \min(\tau_{y_*}, \, \tau_{\text{ref}})$.
            \State For $s \in [t, t + \tau)$, set $P_s = P_t + (s-t) A$ and $X_s = X_t$.
            \If{$\tau_{\text{ref}}< \tau_{y_*}$}
            \State Sample $P_{t + \tau} \sim \rho$ and go to $t = t + \tau$.
            \Else
            \State Set $X_{t + \tau} = y_*$ and go to $t = t + \tau$.
            \EndIf
        \EndWhile
    \end{algorithmic}
\end{algorithm}

\subsection{Fast approximation schemes}\label{sec:fast_approximation_schemes}
\Cref{alg:HMC_continuous_momentum} generates discrete Hamiltonian trajectories by evaluating at every iteration the target density in all neighbouring states, while allowing for a single jump on the $x$-coordinate. In this section we discuss various methods for approximately simulating Hamiltonian dynamics with a reduced computational cost compared to \cref{alg:HMC_continuous_momentum}. In the following, we discuss several techniques that can be either combined or used separately. All these techniques discretize the continuous-time discrete Hamiltonian dynamics and generate ``unadjusted'' discretization methods which introduce an error on the invariant measure $\Hat \pi$. The error can be controlled by choosing a small enough step size.

For concreteness, we consider the dynamics in \cref{eq:discrete_hamiltonian_SDE} with a Gaussian momentum distribution $\rho$ and $(\mu, Q)$ as constructed in~\cref{subsec:factorized_generator}, that is, 
$\mu_x(p) =\mu_x= \sum_{y \neq x} \Bar{Q}_{x,y} \sigma_{x,y}$ and
$Q_{x, y}(p) = \Bar{Q}_{x,y} \max(0, \langle \sigma_{x,y}, p \rangle )$ for $x \neq y$.

\subsubsection{Euler and leapfrog discretizations}
Take a step size $h>0$.
Define 
$K_p^h=e^{hQ(p)}$ to be the transition probability matrix of the CTMC on $\sX$ with generator $Q(p)$ and fixed $p$. 
Then one could consider an explicit discretization of \cref{eq:discrete_hamiltonian_SDE}: given $(x^{(0)},p^{(0)})$, for $n \geq 0$ let
\begin{align}
p^{(n+1)}&\gets  p^{(n)}+h \mu_{x^{(n)}}\label{eq:euler1}\\
x^{(n+1)}&\sim K^h_{p^{(n)}}(x^{(n)},\cdot),\label{eq:euler2}
\end{align}
be the \emph{Euler discretization} and 
\begin{align}
p^{(n+1/2)}&\gets  p^{(n)}+h/2 \mu_{x^{(n)}}\label{eq:leap1}\\
x^{(n+1)}&\sim K^h_{p^{(n+1/2)}}(x^{(n)},\cdot)\label{eq:leap2}\\
p^{(n+1)}&\gets  p^{(n+1/2)}+h/2 \mu_{x^{(n+1)}}\label{eq:leap3}\,
\end{align}
be the \emph{leapfrog discretization}.
Note that simulating $(x^{(n)},p^{(n)})$ given $(x^{(0)},p^{(0)})$ according to \cref{eq:euler1}-\cref{eq:euler2} or to \cref{eq:leap1}-\cref{eq:leap3}
requires the same number of target evaluations, which consists in the computation of the rows of $\bar{Q}$, i.e., $(\bar{Q}_{x,y})_{y\in\sX}$, for each state $x$ visited by the chain.
Given those, computing $\mu_{x^{(i)}}$ and sampling from $K^h_{p^{(i+1/2)}}$ require only computing vector sums and sampling from normalized vectors, which are usually computationally cheaper operations. In this sense, we can assume that, like in the classical continuous case, the computational cost of the Euler and leapfrog integrators is comparable. Similar schemes were considered in  \cite{bertazzi2025piecewise} for discretizing piecewise deterministic Markov processes. We refer to this work also for corresponding error bounds on $\hat \pi$.

\subsubsection{$\tau$-leaping}
Simulating from $K^h_{p}$ exactly requires using the Gillespie algorithm, which can be expensive because it involves computing a new row of $Q(p)$ after each CTMC jump. In many cases, we can use the so-called $\tau$-leaping approximation of $K^h_p$ to reduce computation cost. Roughly speaking, $\tau$-leaping is an approximate discretization method for CTMCs that operates on Cartesian product spaces and moves only one coordinate at a time.
In such a setting, $\tau$-leaping relies on ``factorized approximations'', i.e., it approximates the CTMC dynamics by letting each coordinate evolve independently for a time interval of length $h$ as if the other coordinates were fixed.

Specifically, let $\sX=\prod_{i=1}^n\sX_i$, with e.g.\ $\sX_i=\{0,1\}$ or $\sX_i=\{1,\dots,K\}$. 
Then we approximate 
\begin{align}\label{eq:tau_leap}
K^h_p(x,y)\approx \tilde{K}^h_p(x,y)\coloneq \prod_{i=1}^n K^h_{p,i}(x_i,y_i;x_{-i}),
\end{align}
where $K^h_{p,i}(\cdot,\cdot;x_{-i})=e^{hQ^{(i)}(p,x_{-i})}$ is the transition matrix of a CTMC on $\sX_i$ with generator $
Q^{(i)}(p,x_{-i})=(Q^{(i)}_{x_i,y_i}(p,x_{-i}))_{x_i,y_i\in\sX_i}$ defined by
$Q^{(i)}_{x_i,y_i}(p,x_{-i})=Q_{(x_i,x_{-i}),(y_i,x_{-i})}(p)$.
\Cref{alg:splitting} combines the leapfrog discretization scheme with $\tau$-leaping to approximate discrete Hamiltonian dynamics in the hypercube $\sX = \{0,1\}^n$. If, for example, $\sX = \Z^n$, then lines 5-8 would be replaced with $N_i \sim \text{Poiss}(h\, \lambda_i(X_t, P_{t + h/2}))$ and $A = A \circ R_i^{N_i}$ where  $R_i^{N_i}$ is the $N_i$-fold  composition of the action $R_i$. 

The computational advantage of $\tau$-leaping stems from the fact that the coordinate-wise transitions can be simulated simultaneously. By contrast, \cref{alg:HMC_continuous_momentum} simulates only one transition at each iteration.

\begin{algorithm}
\caption{\label{alg:splitting} Leapfrog integrator with $\tau$-leaping in the hypercube}
    \begin{algorithmic}[1]
        \Statex Requires: $\sX = \{0,1\}^n$. A set of $n$ commuting actions $R_i \colon \mathcal{X} \to \mathcal{X}$ with rates $\lambda_i(x,p)$. A step size $h>0$. 
        \Statex Initialize $t=0$ and $(X_t,P_t) \in \sX \times \R^d$.
        \While{Stopping criterion is not met}
            \State $P_{t + h/2} = P_t + h/2 \, \mu_{X_t}$
            \State Initialize container of actions $A = \Id$
            
            \For{$i \in [n]$} \Comment{This loop can be executed in parallel.}
                \State $\tau_i \sim \text{Exp}(\lambda_i(X_t, P_{t + h/2}))$ 
                \If {$\tau_i < h$}
                    \State $A = A \circ R_i$
                \EndIf
            \EndFor
            \State $X_{t + h} = A(X_t)$ 
            \State $P_{t + h} = P_{t + h/2} + h/2\, \mu_{X_{t + h}}$
            \State $t = t + h$.
        \EndWhile
    \end{algorithmic}
\end{algorithm}

\subsubsection{Gradient approximations via embeddings}
The main computational cost of both \cref{alg:HMC_continuous_momentum} and \cref{alg:splitting} is usually the evaluation of $\bar{Q}_{x,y}=G(\pi_y /\pi_x )$ for $y$ in the neighbourhood of $x$, i.e., $y$ s.t.\ $\sigma_{x,y}\neq 0$.
When $\sX=\{0,1\}^n$, following \cite{grathwohl2021oops}, many authors consider approximating such ratios with $\pi_y /\pi_x \approx \exp(\langle \nabla \log \check \pi(x), y-x \rangle)$ with $\log\check \pi:\R^n\to\R$ being some continuous function that coincides with $\log\pi$ on $\{0,1\}^n\subseteq \R^n$. This way we obtain an approximation to $\pi_y /\pi_x $ for the $n$ neighbours of $x$ (i.e., for the Hamming ball of radius one) with a single gradient computation.

\section{Numerical illustrations} \label{sec:numerics}

We now numerically illustrate the behaviour of the discrete HMC algorithm and its fast approximations described above.

\subsection{Diffusive-to-ballistic speed-up}\label{sec:diffusive-to-ballistic-speed-up}
The motivation for this work is to accelerate the convergence of reversible Markov chain Monte Carlo methods to $\pi$ for ill-conditioned targets. We showcase this by comparing exact discrete HMC and the base reversible process, a Markov process which uses the reversible Markov kernel $e^{t\bar Q}$, for the examples described in \cref{sec:specific_examples}. The performance of the two algorithms is always compared at equal number of target evaluations.

In the discrete hypercube (\cref{sec:example_discrete_hypercube}), we consider exact discrete HMC (\cref{alg:HMC_continuous_momentum}) with $\sigma_{x,y}$ as in \cref{eq:1d_momentum_hypercube}, the constrained target in \cref{eq:target_with_hard_constraints} with $m = 200$, and velocity refreshment at every $1/\gamma$ time units, with $\gamma = 1/m$. The top panels in \cref{fig:ballistic_speedup} show the traces and the autocorrelation functions of the slow direction  $z = m^{-1} |x|_2$.

For multivariate categorical spaces $\mathcal{X} = [K]^n$ (\cref{sec:example_categorical_spaces}), we consider a posterior target distribution given by the Bayesian Gaussian mixture model
\begin{equation}
\label{eq:gaussian_mixture_model}
        y_i \mid \theta, w, x \, \overset{i.i.d}{\sim} \sN(\theta_{x_i}, \sigma^2), \qquad
        x_i\mid \theta, w \, \overset{i.i.d}{\sim} \mathrm{Cat}(w)\\ 
    \qquad w\sim\text{Dir}(\alpha), \qquad \theta_j \overset{i.i.d}{\sim} \sN(\theta_0, \sigma^2_0),
\end{equation}
for $i = 1,2,\ldots, n$ and $j = 1, 2, \ldots,K$. In this model, $y_1,\dots,y_n$ are the observations, $\theta$, $w$, $x$ are unknown parameters, whereas $\sigma>0, \sigma_0>0, \alpha \in (0,\infty)^K, \theta_0 \in \R$ are fixed parameters. In this setting, the parameters $\theta, w$ can be integrated out, yielding a posterior $\pi=(\pi_x)_{x\in[K]^n}$.

We set $K = 2$ and $n = 500$, draw $y_i\overset{i.i.d}{\sim} 0.8 \sN(1,1) + 0.2 \sN(-1,1)$, and set $\alpha_j = 1/2$, $\sigma = \sigma_0 = 1$ and $\theta_0 = 0$. We run exact discrete Hamiltonian dynamics with $\sigma_{x,y}$ as in \cref{eq:sigma_categorical_space} and full velocity refreshments at every $1/\gamma$ unit times, where $\gamma = 1/n$.
The bottom panels of \Cref{fig:ballistic_speedup} show the traces and autocorrelation functions for the test function $n_{\max_k}(x) = \max_{j=1,\dots,K}{g_j(x)}$ where $g_j(x) = \sum_{i=1}^n \1(x_i = j)$, representing the estimated size of the largest cluster.

\begin{figure}[!ht]
    \centering
\includegraphics[width=0.9\linewidth]{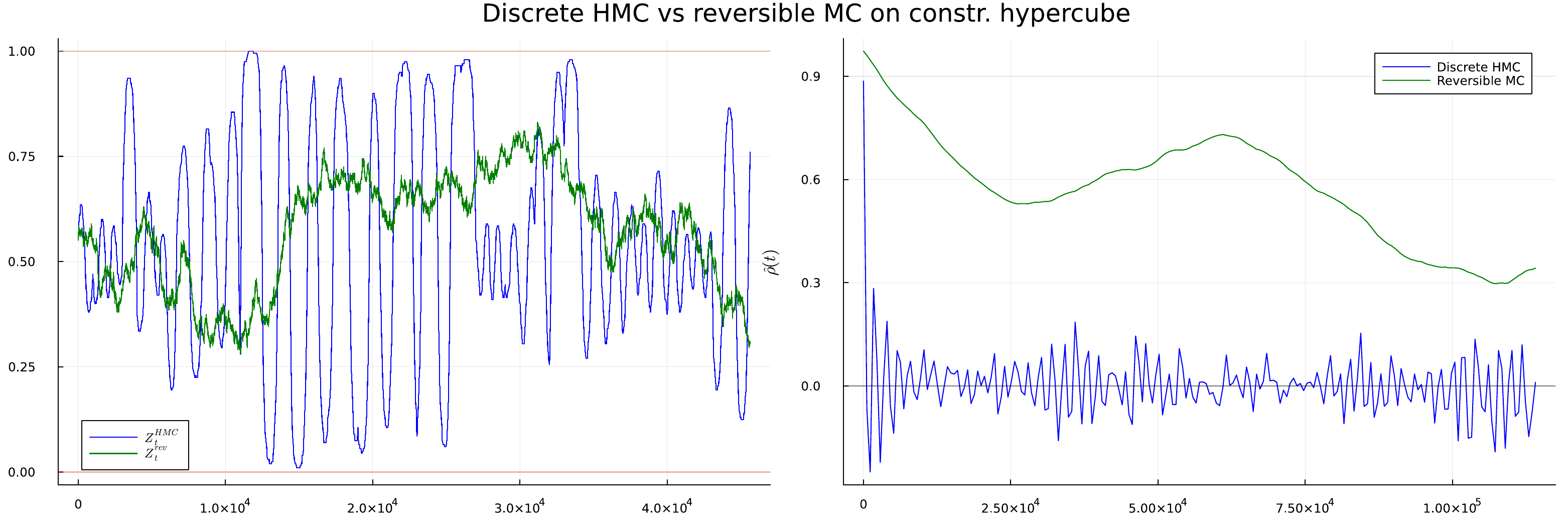}
\includegraphics[width=0.9\linewidth]{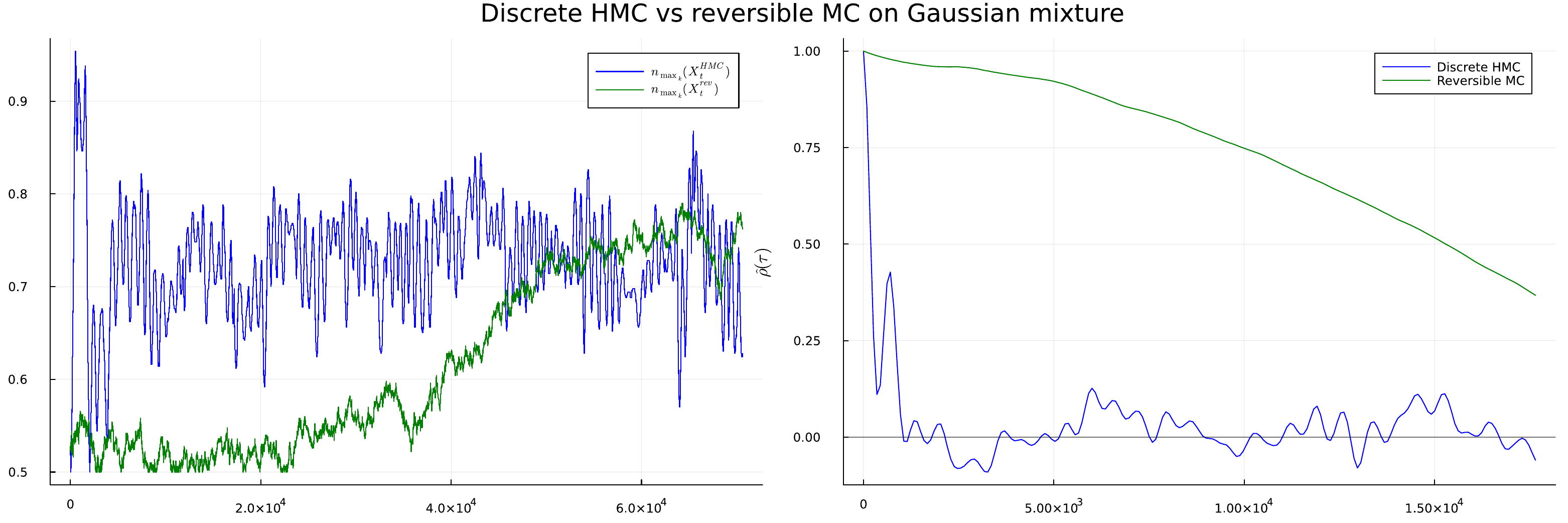}
\caption{Traces and autocorrelation functions of discrete Hamiltonian Monte Carlo and the base reversible process for the constrained hypercube target (\cref{eq:target_with_hard_constraints}, top panels) and the posterior of the Gaussian mixture model (\cref{eq:gaussian_mixture_model}, bottom panels).}
    \label{fig:ballistic_speedup}
\end{figure}

Finally, we consider a ranking problem on the space of permutations over $n$ elements, i.e., the symmetric group $\Sigma_n$ (c.f. \cref{sec:example_permutations_ranking}), and the  Plackett–Luce target \cite{plackett1975analysis,luce1959individual} given by 
\begin{align}
    \label{eq:plackett-luce}
    \pi_x
& =
\prod_{r=1}^{n-1}
\frac{\exp(w_{x(r)})}
{\sum_{\ell=r}^{n} \exp(w_{x(\ell)})},
& x \in \Sigma_n,
\end{align}
for some $w \in \R^n$. We set $n = 200$ and fix $w_{1},w_2,\ldots,w_n \overset{i.i.d}{\sim} \sN(0,1)$.
\Cref{fig:non_rev_vs_rev_ranking} shows the traces of the index position $\{j \colon x(j) = i\}$ for those indices $i$ such that $w_i$ is the minimum, maximum and median value of $w$.

In all these numerical simulations, it is qualitatively clear that discrete HMC mixes much faster than the base reversible Markov process.

\begin{figure}[!ht]
    \centering
    \includegraphics[width=0.9\linewidth]{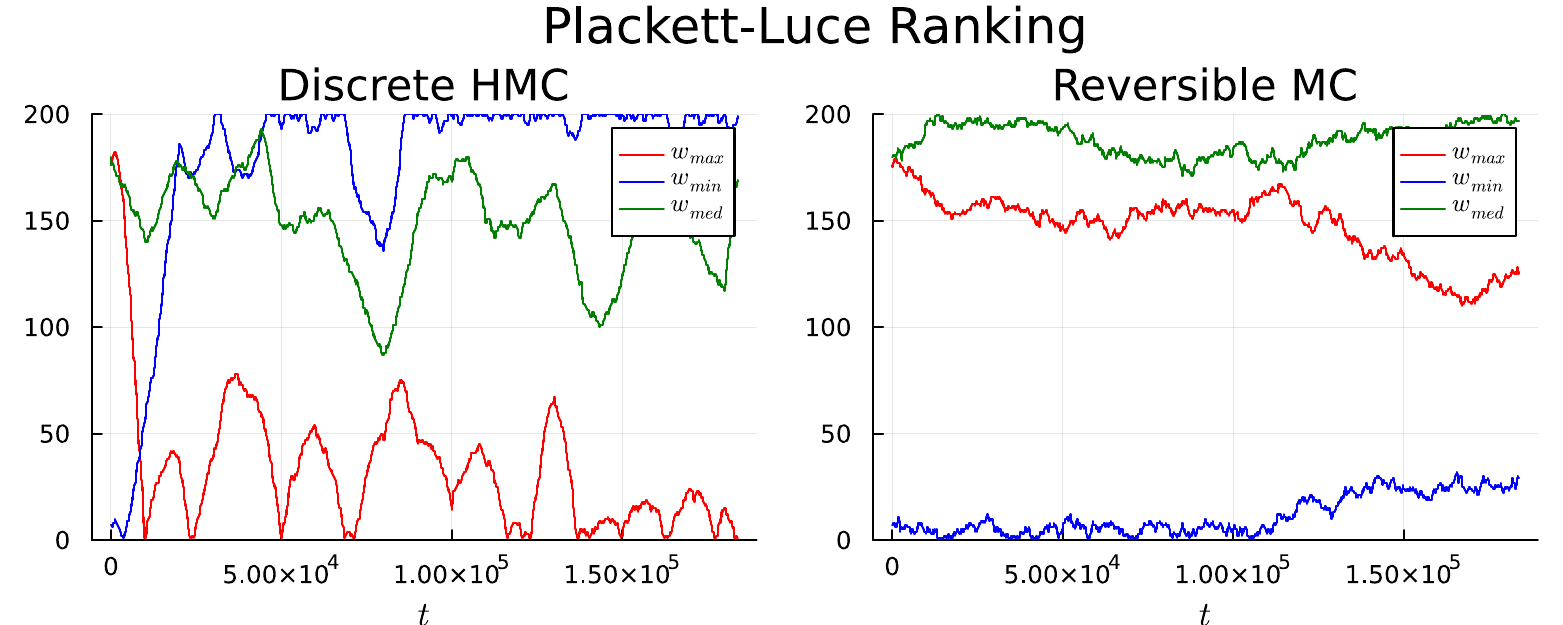}
    \caption{Traces of the index position of the maximum (red), median (green) and minimum (blue) weight $w$. Exact Hamiltonian dynamics (left) and the base reversible algorithm (right). The target is given by~\cref{eq:plackett-luce}.}
    \label{fig:non_rev_vs_rev_ranking}
\end{figure}

\subsection{Hamiltonian flow and discretization error on the hypercube}

\Cref{theo:scaling_limits_constrained_hypercube} shows that functionals of our discrete Hamiltonian dynamics in the hypercube converge to a deterministic flow. This flow coincides with the classical Hamiltonian dynamics \eqref{eq:generator_HMC_and_ULD}, after a change of variable  (\cref{remark:hamiltonian_dynamics_scaling_hypercube}).

This is illustrated in \cref{fig:scaling_limits_hypercube}, where we show, for the discrete Hamiltonian dynamics and their Euler and leapfrog discretizations with $\tau$-leaping discussed in \cref{sec:fast_approximation_schemes}, the evolution of $z = |x|_2/m$ and $p$ in the constrained hypercube \cref{eq:target_with_hard_constraints}, with $m = 2000$, and $\bar Q$ as in \cref{eq:Q_bar_LB}. This illustration is insightful since it makes clear that discrete Hamiltonian dynamics (when $\gamma = 0$)  on  $(z, p)$ are almost periodic, with limiting deterministic dynamics given in \cref{eq:scaling_limits_constr.hypercube}.
However, the periodicity can be broken by introducing velocity refreshments at every $1/\gamma$ unit times (right panels).  Furthermore, akin to the behaviour of classical HMC in continuous space, \cref{fig:scaling_limits_hypercube} shows that the error in the Hamiltonian $\mathcal{H}(z,p) = |p|^2/2 - \log h(z)$, which is the quantity preserved by the limiting ODE in \cref{eq:scaling_limits_constr.hypercube}, accumulates faster over time under the Euler scheme, compared to the leapfrog scheme with the same discretization step size $h$.

\begin{figure}[!ht]
    \centering
    \includegraphics[width=0.9\linewidth]{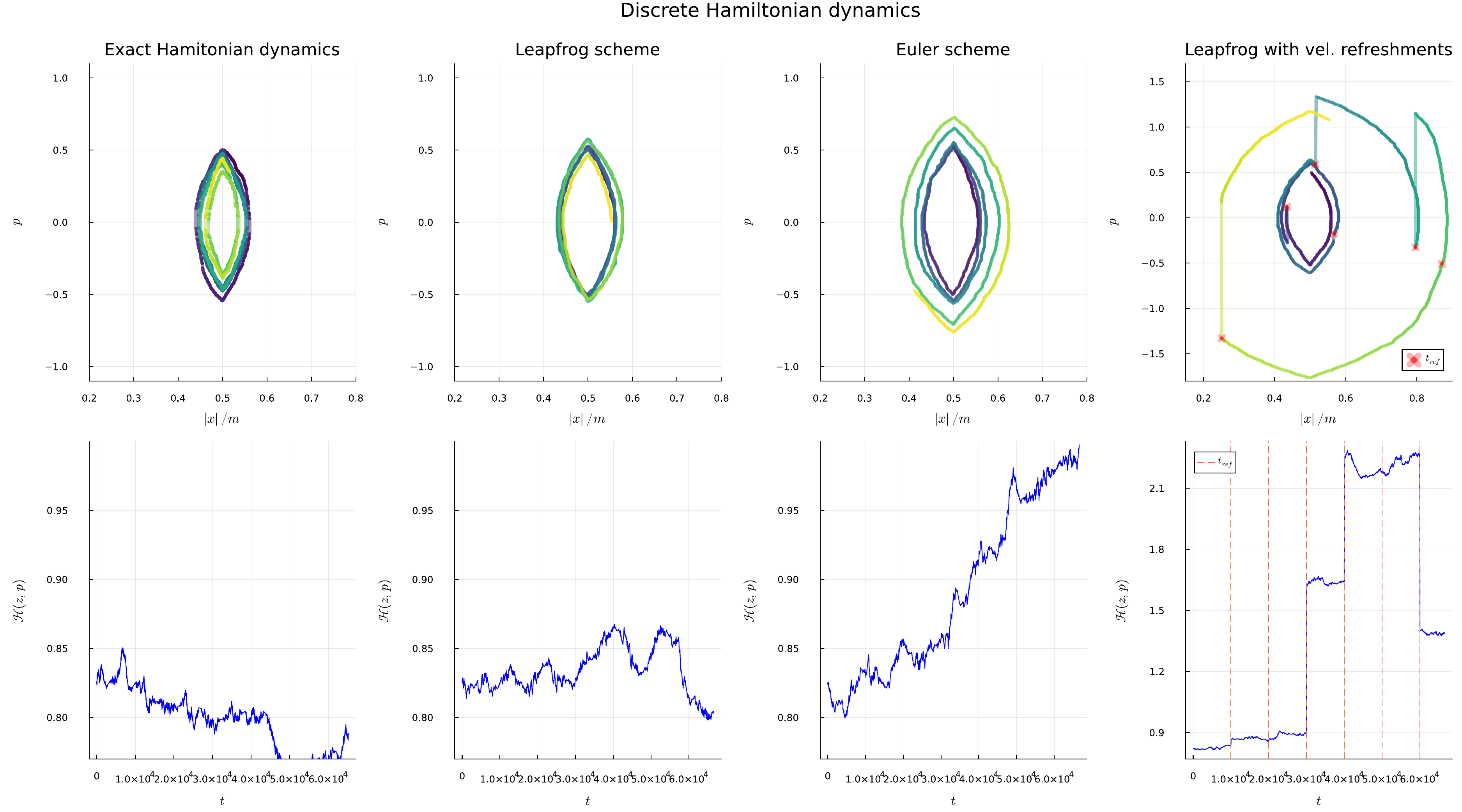}
    \caption{Discrete Hamiltonian dynamics on the hypercube with the target in \cref{eq:target_with_hard_constraints} and $m = 2000$. Top panels: $(|x|/m, p)$ when discretizing the dynamics with Euler, leapfrog and leapfrog with velocity refreshments. Bottom panels: corresponding trace of the Hamiltonian $\mathcal{H}(z,p) = |p|^2/2 - \log h(z)$.}
    \label{fig:scaling_limits_hypercube}
\end{figure}

\subsection{Discretization error of $\tau$-leaping}
Here we numerically test the error on the invariant measure $\Hat \pi$ incurred when using either Euler or leapfrog discretization schemes together with $\tau$-leaping for the examples and targets considered in \cref{sec:diffusive-to-ballistic-speed-up}.

We simulate our approximate algorithm (e.g.\ \cref{alg:splitting}) varying the step size $h$. We then compute total-variation (TV) errors, the estimated standard deviation of the velocity component $p$ (respectively $\|p\|$, when $p$ is multidimensional) and the speed-up relative to the exact dynamics, which is here measured as the average number of simultaneous changes on the $x$ components performed in one single iteration of the algorithm. This is justified since exact Hamiltonian dynamics change only one single component for every iteration and one iteration of either algorithms uses the same number of target evaluations. The TV error is computed for the same test functions considered in \cref{sec:diffusive-to-ballistic-speed-up}, that is $m^{-1} |x|_2$ for the hypercube, $n_{\max_k}(x)$ for the mixture model and $i \colon x(i) = j$, where $j$ is the index of the median element of $w$. In all examples, we reduced the dimensionality of the target to get more accurate results. In particular, we set $m  = 50$ for the hypercube example, $n = 100$ for the Gaussian mixture model and $n = 50$ for the example on the space of permutations. 

In the hypercube, the invariant distribution is known analytically, while for other examples it is computed with a long run of the exact Hamiltonian dynamics. In all cases, the invariant measure was discretized over $10$ equally sized bins before computing the total-variation distance. \Cref{fig:discretization_error} shows the results for both Euler and leapfrog schemes. As expected, leapfrog tends to have smaller error than Euler schemes, in particular on the constrained hypercube.

\begin{figure}[!ht]
    \centering
    \includegraphics[width=0.75\linewidth]{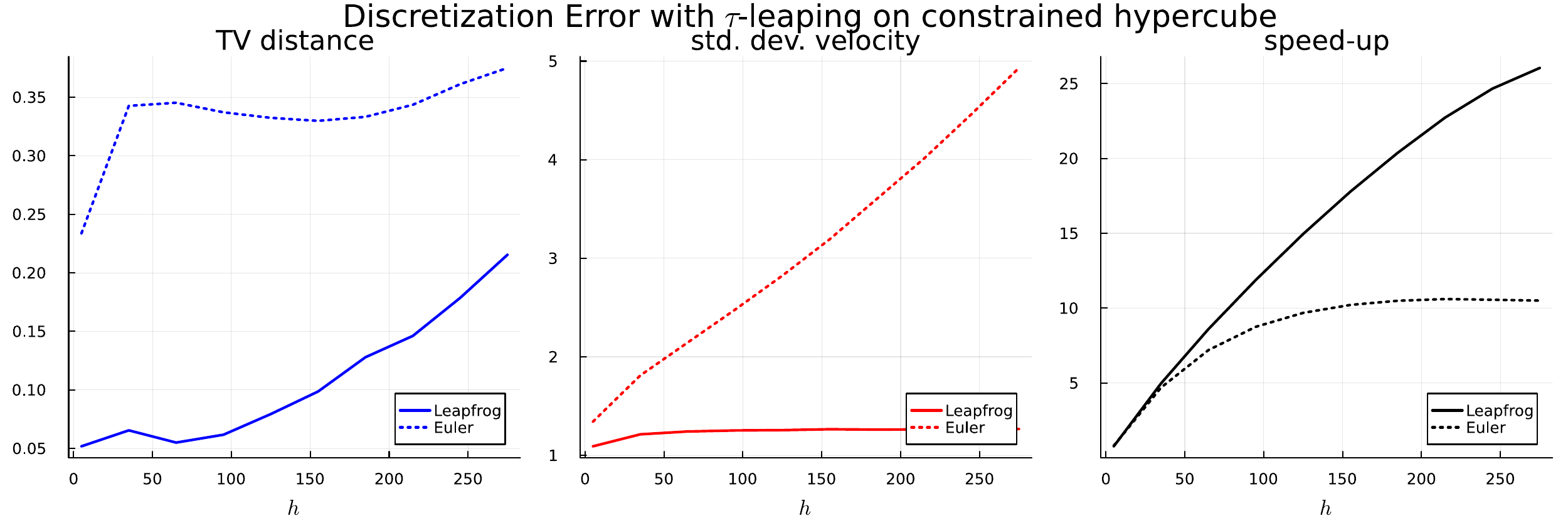}
     \includegraphics[width=0.75\linewidth]{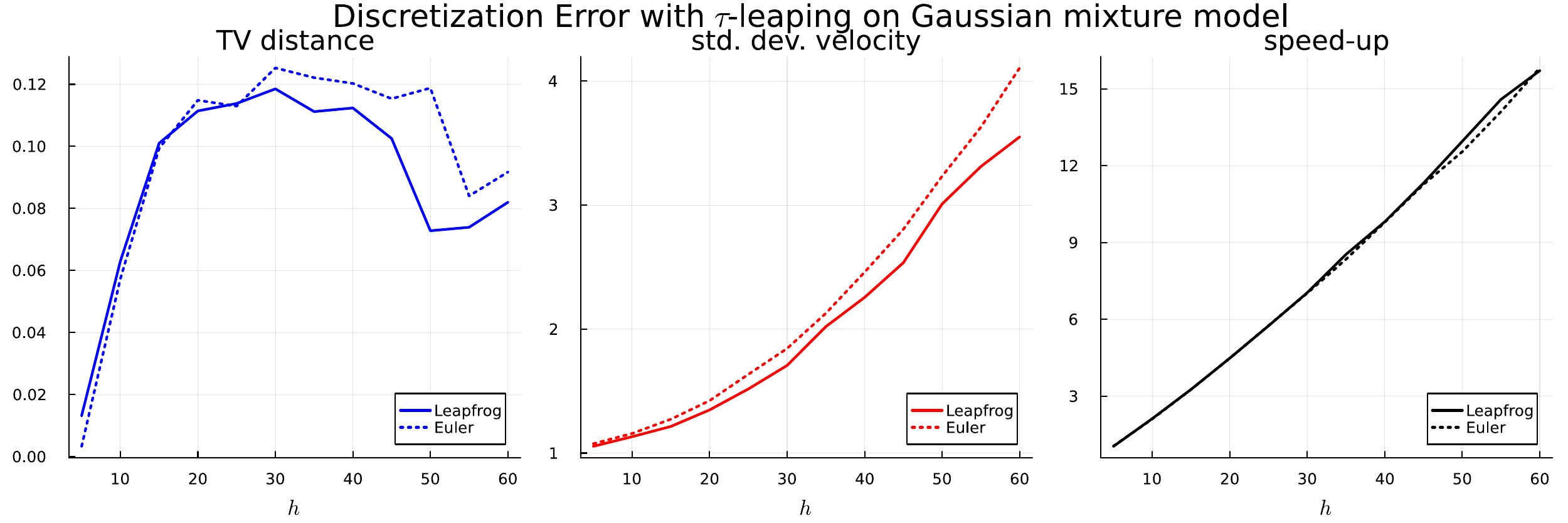}
     \includegraphics[width=0.75\linewidth]{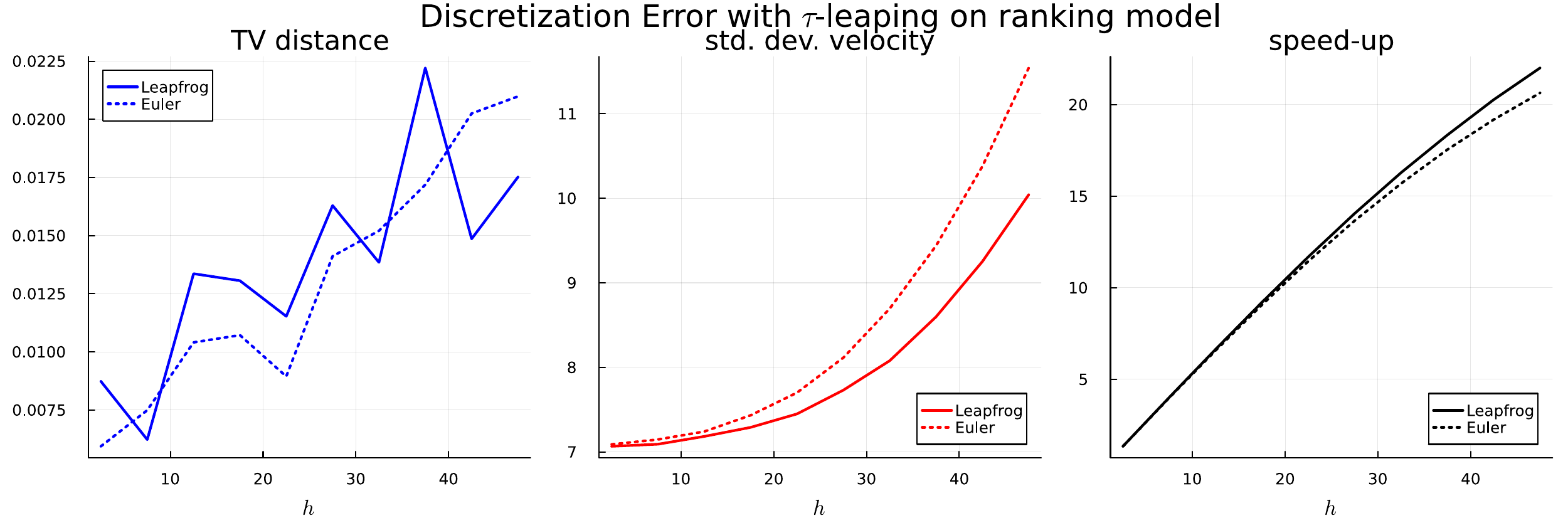}
    \caption{Total variation distance (blue),  estimated standard deviation of $p$ (red, top and middle panels) and $\|p\|$ (red, bottom panel) and speed-up (black) for different values of step size $h$ ($x$-axis) for  Euler (dashed lines) and leapfrog (solid lines). The examples considered are the constrained hypercube (top panels), Gaussian mixture (middle panels) and Plackett-Luce (bottom panels).}
    \label{fig:discretization_error}
\end{figure}

\section{Conclusion}

\subsection{Summary}

In this paper, we proposed a broadly applicable framework for constructing non-reversible sampling algorithms on discrete spaces.
An interesting feature of our proposed discrete HMC algorithm is its adaptability across a variety of applications, offering advantages in particular when there is prior information about the problem (e.g.\ slow directions).

We studied both theoretically and numerically the properties of the corresponding Markov processes. Interestingly, our scaling limit results in~\cref{sec:scaling_limits_infill,sec:scaling_limit_hypercube} show that discrete Hamiltonian dynamics can exhibit a ballistic behaviour even in regimes where simpler and more standard non-reversible MCMC methods lose their ballistic property.
Such a diffusive-to-ballistic speed-up is also supported by numerical experiments in~\cref{sec:numerics}.

\subsection{Future research directions}\label{sec:future_res}

The results presented in this paper suggest several possible directions for future work.
\begin{itemize}
    \item 
    We conjecture that a scaling limit result similar to~\cref{theo:scaling_limits_constrained_hypercube} should also hold for the discrete Hamiltonian dynamics with a multi-dimensional velocity when accelerating the process by an additional factor of $\sO(\sqrt{m})$. However, we were not able to come up with a proof. A possible approach is to make use of more involved stochastic homogenization arguments (see e.g.\ \cite{pavliotis2008multiscale}) to quantify the averaging of the velocity components.

    \item 
    \cref{prop:spectral_gap} provides a contraction rate in $L^2$ for the discrete Hamiltonian dynamics which translates into an upper bound on the relaxation time. This upper bound does not match the lower bound in~\cref{prop:relaxation_time}.
    In contrast, several works have established optimal quantitative convergence rates for the underdamped Langevin diffusion~\cite{cao2023explicit,eberle2024non} or for some classes of PDMPs~\cite{lu2022explicit,eberle2025convergence}.
    Although we identified structural differences between these processes and the discrete Hamiltonian dynamics (e.g.\ in \cref{sec:comparison_pdmps}), it would be natural to investigate whether these proof techniques can be adapted to our case.
    Our scaling limits in \cref{sec:scaling_limit_hypercube_1d} offer a first step in this direction, but they are currently limited to toy target distributions.

    \item An important practical limitation of our proposed discrete HMC  (\cref{alg:HMC_continuous_momentum}) is its computational cost, because the algorithm requires evaluating the density on the whole neighbourhood of the current state at each iteration.
    Several possible numerical schemes were proposed in~\cref{sec:fast_approximation_schemes} in order to accelerate computations at the cost of introducing a bias on the invariant measure. It would thus be interesting to investigate different discretization schemes of discrete Hamiltonian dynamics which are ``uninformed'', that is, which only require $\sO(1)$ target evaluations per iteration, while preserving the invariant measure and retaining a ballistic behaviour.
    \item Finally, we did not investigate adapting the magnitudes of the directions $\sigma_{x,y}$. Analogously to the mass matrix in classical HMC \cite{neal2011mcmc}, such adaptation could reduce anisotropy across directions.
\end{itemize}

\clearpage

\section*{Acknowledgments}
All authors acknowledge support from the European Research Council (ERC) through the Starting Grant ‘PrSc-HDBayLe’, project number 101076564.


\bibliographystyle{imsart-number} 
\bibliography{bibliography}       


\newpage

\appendix

\section{Reduction to low-dimensional projections}\label{sec:appendix_proj}
The following lemma, which relates to the so-called notion of lumpability~\cite{buchholz1994exact} and deinitializing Markov chains~\cite{roberts2001markov}, is used in~\cref{sec:scaling_limit_hypercube} to prove~\cref{prop:non_rev_unif}.

\begin{lm}\label{lemma:reduction}
Let $(X_t)_{t=0,1,2,\ldots}$ be a $\pi$-invariant Markov chain on a finite space $\mathcal{X}$
with transition kernel $K=(K_{x,y})_{x,y\in\mathcal{X}}$. Then:
\begin{enumerate}
    \item[(a)] If
$
\phi:\mathcal{X}\to E$
is such that
\begin{align}\label{eq:Markov_proj}
\bar{K}_{a,b}&\coloneq\sum_{y\in\phi^{-1}(b)}K_{x,y} ,
& a,b\in E,\,\phi(x)=a ,
\end{align}
depends on $x$ only through $\phi(x)$, then $(\phi(X_t))_{t=0,1,2,\ldots}$ is a Markov chain on $E$, with transition kernel 
$\bar{K}=(\bar{K}_{a,b})_{a,b\in E}$.
\item[(b)] 
Assume in addition that $\bar{\pi}_a > 0$ for every $a \in E$, and that
\begin{equation}\label{eq:intertwining}
\bar{\pi}_{\phi(y)}\sum_{x\in \phi^{-1}(a)}\pi_x\,K_{x,y}
=
\bar{\pi}_a\bar K_{a,\phi(y)}\,\pi_y
\end{equation}
for every $y\in\sX$ and $a\in E$, with $\bar{\pi}=\phi_{\#}\pi$.
Then, for every $t\ge0$ and every $\mu\in\sP(\sX)$ satisfying
$\bar{\pi}_{\phi(x)}\mu_x
=
\bar\mu_{\phi(x)}\,\pi_x$ for all $x\in\sX$, with $\bar{\mu}=\phi_{\#}\mu$, we have
\[
\|\mu K^t-\pi\|_{\mathrm{TV}}
=
\|\bar\mu\bar K^t-\bar\pi\|_{\mathrm{TV}}.
\]
\end{enumerate}
\end{lm}

\begin{proof}
Part (a) is classical; see e.g.\ \cite{buchholz1994exact}, and we omit its proof for brevity. 
Consider now part (b). 
By the assumption on $\mu$, we have
\begin{align*}
(\mu K)_y
&=
\sum_{x\in\sX}\mu_x K_{x,y}
=
\sum_{a\in E}
\sum_{x\in \phi^{-1}(a)}
\frac{\bar{\mu}_a\pi_x}{\bar{\pi}_a} K_{x,y}
=
\sum_{a\in E}
\frac{\bar{\mu}_a}{\bar{\pi}_a}
\sum_{x\in \phi^{-1}(a)}
\pi_x K_{x,y}\,.
\end{align*}
Combining the latter with \eqref{eq:intertwining} we obtain
\begin{align*}
\bar{\pi}_{\phi(y)}(\mu K)_y
=
\sum_{a\in E}
\frac{\bar{\mu}_a}{\bar{\pi}_a}
\bar{\pi}_{\phi(y)}
\sum_{x\in \phi^{-1}(a)}
\pi_x K_{x,y}
\stackrel{\eqref{eq:intertwining}}=
\sum_{a\in E}
\bar{\mu}_a\bar K_{a,\phi(y)}\,\pi_y
=
(\bar{\mu}\bar K)_{\phi(y)}\,\pi_y\,.
\end{align*}
By induction, the above argument implies 
$\bar{\pi}_{\phi(y)}(\mu K^t)_y
=
(\bar{\mu}\bar K^t)_{\phi(y)}\,\pi_y$
for every $t=0,1,2,\dots$.
The result then follows from
\begin{align*}
\|\mu K^t-\pi\|_{\mathrm{TV}}
&=
\frac12
\sum_{b\in E}
\sum_{y\in \phi^{-1}(b)}
\left|
(\bar\mu\bar K^t)_b\,
\pi_y/\bar{\pi}_b
-
\bar\pi_b\,
\pi_y/\bar{\pi}_b
\right|\\
&=
\frac12
\sum_{b\in E}
\left|
(\bar\mu\bar K^t)_b-\bar\pi_b
\right|
\sum_{y\in \phi^{-1}(b)}
\pi_y/\bar{\pi}_b \\
&=
\frac12
\sum_{b\in E}
\left|
(\bar\mu\bar K^t)_b-\bar\pi_b
\right| \\
&=
\|\bar\mu\bar K^t-\bar\pi\|_{\mathrm{TV}}\,.
\end{align*}
\end{proof}

\end{document}